\documentclass[12pt,reqno]{amsart}
\advance \topmargin by -\headheight
\advance \topmargin by -\headsep
\evensidemargin \oddsidemargin
\usepackage{amsmath}
\usepackage{amsfonts}
\usepackage{stmaryrd}
\usepackage{amssymb}
\usepackage{csquotes}
\usepackage{amsthm,thmtools}

\usepackage{mathrsfs}
\usepackage{dsfont}
\usepackage{bm}
\usepackage{cite}
\usepackage{soul}

\numberwithin{equation}{section}
\renewcommand\vec{\bm}
\newcommand{\n}[1]{\|{#1}\|}

\newtheorem{theorem}{Theorem}[section]

\newtheorem{lemma}[theorem]{Lemma}
\newtheorem{Proposition}[theorem]{Proposition}
\newtheorem{Conjecture}[theorem]{Conjecture}
\newtheorem{Corollary}[theorem]{Corollary}

\usepackage{mathtools}
\DeclarePairedDelimiter{\ceil}{\lceil}{\rceil}

\title[Unbounded expansion of polynomials and products]{Unbounded expansion of polynomials and products}
\author[Akshat Mudgal]{Akshat Mudgal}
\address{Mathematical Institute, University of Oxford, Oxford OX2 6GG, UK}
\email{mudgal@maths.ox.ac.uk}


\renewcommand\vec{\bm}

\begin{document}

\begin{abstract}
Given $d,s \in \mathbb{N}$, a finite set $A \subseteq \mathbb{Z}$ and polynomials $\varphi_1, \dots, \varphi_{s} \in \mathbb{Z}[x]$ such that $1 \leq \deg \varphi_i \leq d$ for every $1 \leq i \leq s$,
we prove that
\[ |A^{(s)}| + |\varphi_1(A) + \dots + \varphi_s(A) | \gg_{s,d} |A|^{\eta_s} , \]
 for some $\eta_s \gg_{d} \log s / \log \log s$. Moreover if $\varphi_i(0) \neq 0$ for every $1 \leq i \leq s$, then 
\[ |A^{(s)}| + |\varphi_1(A) \dots \varphi_s(A) | \gg_{s,d} |A|^{\eta_s}. \]
These generalise and strengthen previous results of Bourgain--Chang, P\'{a}lv\"{o}lgyi--Zhelezov and Hanson--Roche-Newton--Zhelezov. We derive these estimates by proving the corresponding low-energy decompositions. The latter furnish further applications to various problems of a sum-product flavour, including questions concerning large additive and multiplicative Sidon sets in arbitrary sets of integers.

\end{abstract}

\maketitle


\section{Introduction}

A central theme in arithmetic combinatorics is the interplay of addition and multiplication over arbitrary sets of integers. In particular, a fundamental phenomenon in the area, which is exhibited and quantified in a rich collection of results, suggests that additive and multiplicative structures in finite sets of integers do not coincide with each other. Thus, given a natural number $s$ and a finite set $A \subseteq \mathbb{R}$, we define the $s$-fold sumset $sA$ and the $s$-fold product set $A^{(s)}$ of $A$ to be
\[ sA = \{ a_1 + \dots + a_s \ | \ a_1, \dots, a_s \in A\}  \ \text{and} \  A^{(s)} = \{ a_1 \dots a_s \ | \ a_1, \dots , a_s \in A \} \]
respectively. These can be seen to be standard measures of arithmetic structure, since whenever $A$ is an arithmetic progression, then $|A| \leq |sA| \ll_{s} |A|$, while whenever $A$ is a geometric progression, we have that $|A| \leq |A^{(s)}| \ll_{s} |A|$. Quantifying the aforementioned philosophy of the lack of coexistence between additive and multiplicative structure, Erd\H{o}s and Szemer\'{e}di \cite{ES1983} proposed the following conjecture. 

\begin{Conjecture} \label{esc}
For any $s \in \mathbb{N}$ and $\varepsilon >0$ and finite set $A \subseteq \mathbb{Z}$, we have
\[|sA| + |A^{(s)}| \gg_{s, \varepsilon} |A|^{s - \varepsilon} . \]
\end{Conjecture}

While this conjecture remains wide open even in the case when $s=2$, there have been some breakthrough results in this direction, including the beautiful work of Bourgain--Chang \cite{BC2004} that delivers the bound
\begin{equation} \label{bclb}
 |sA| + |A^{(s)}| \gg_{s} |A|^{(\log s)^{1/4}}
 \end{equation}
in the setting of Conjecture $\ref{esc}$.  We refer to this type of an estimate as exhibiting unbounded expansion since the exponent $(\log s)^{1/4} \to \infty$ as $s \to \infty$.  Furthermore,  this has since been quantitatively improved only once, wherein the exponent $(\log s)^{1/4}$ has been upgraded to $(\log s)^{1 - o(1)}$  by P\'{a}lv\"{o}lgyi--Zhelezov \cite{PZ2020}. 
\par

The sum-product phenomenon has expanded vastly since the work of Erd\H{o}s and Szemer\'{e}di \cite{ES1983}, and now encompasses a variety of results which highlight an incongruence between many different types of arithmetic structures. Despite this, there has only been one other result which allows for unbounded expansion, which is due to Hanson, Roche-Newton and Zhelezov \cite{HRZ2020}. In particular, they showed that for any finite set $A$ of integers and for any $u \in \mathbb{Z} \setminus \{0\}$ and $s\in \mathbb{N}$, one has
\begin{equation} \label{hrnz1}
 |A^{(s)}| + |(A+u)^{(s)}| \gg_{s} |A|^{(\log s)^{1/2 - o(1)}} .
 \end{equation}
\par

Recently, an inquiry into a stronger version of sum-product type results was put forth by Balog--Wooley \cite{BW2017}. In order to present this, we first present some definitions, and thus, given $\vec{\varphi} = (\varphi_1, \dots, \varphi_{2s}) \in (\mathbb{Z}[x])^{2s}$, we define the mixed energies
\[ E_{s, \vec{\varphi}}(A) = | \{ (a_1, \dots, a_{2s}) \in A^{2s} \ | \ \varphi_1(a_1) + \dots + \varphi_s(a_s) = \varphi_{s+1}(a_{s+1}) + \dots + \varphi_{2s}(a_{2s}) \}| \]
and
\[ M_{s, \vec{\varphi}}(A) = | \{ (a_1, \dots, a_{2s}) \in A^{2s} \ | \ \varphi_1(a_1)  \dots  \varphi_s(a_s) = \varphi_{s+1}(a_{s+1})  \dots  \varphi_{2s}(a_{2s}) \}| . \]
When $\varphi_1(x) = \dots = \varphi_{2s}(x) = x$ for every $x \in \mathbb{Z}$, we write $E_{s}(A) = E_{s, \vec{\varphi}}(A)$ and $M_{s}(A) = M_{s, \vec{\varphi}}(A)$. A standard application of the Cauchy-Schwarz inequality then gives us that
\begin{equation} \label{csap}
 |sA| \geq |A|^{2s} E_{s}(A)^{-1} \ \ \text{and} \ \ |A^{(s)}| \geq |A|^{2s} M_{s}(A)^{-1},
 \end{equation}
for every $s \in \mathbb{N}$ and every finite $A \subseteq \mathbb{Z}$. Noting this along with Conjecture $\ref{esc}$, it is natural to expect that for every finite $A \subseteq \mathbb{Z}$, one may write $A = B \cup C$, with $B,C$ disjoint,  such that 
\begin{equation} \label{defld}
 E_{s}(B) \ll_{s} |A|^{2s - 1 - c_s} \ \ \text{and} \ \ M_{s}(C) \ll_{s} |A|^{2s - 1 -c_s} ,
 \end{equation}
for some $c_s>0$. This was termed by Balog and Wooley as a \emph{low-energy decomposition}, who proved the first such result with $c_s < 2/33$. In view of Conjecture $\ref{esc}$, they further speculated that one should be able to take $c_s \to \infty$ as $s \to \infty$. Moreover, while many works have subsequently improved admissible exponents in $\eqref{defld}$, till recently, it was still not known whether one may choose $c_s \to \infty$ as $s \to \infty$. In our recent work \cite{Mu2021c}, this speculation was confirmed quantitatively with \cite[Corollary $1.3$]{Mu2021c} delivering the bound $c_s \gg (\log \log s)^{1/2 - o(1)}$. The latter was further quantitatively improved by Shkredov \cite{Shk2022} to $c_s \gg (\log s)^{1/2 - o(1)}$, and it is widely believed that this type of method limits to delivering exponents of the shape $c_s \ll \log s/ \log \log s$, due to examples of the form \cite[Proposition 1.5]{PZ2020}.
\par


Our main aim of this paper is to generalise all of the above sum-product results for additive and multiplicative equations over polynomials, and achieve the speculated exponent of $\log s/\log\log s$ in each of these problems.
This is recorded in our main result below. 

\begin{theorem}  \label{th3}
For any $d,s \in \mathbb{N}$, there exists some $\eta_s \gg_{d} \log s / \log \log s$ such that the following holds true. Any finite set $A \subseteq \mathbb{Q}$ may be written as $A = B \cup C$, for disjoint sets $B,C$, such that for any $\vec{\varphi} \in (\mathbb{Z}[x])^{2s}$ 
satisfying $1 \leq \deg \varphi_1, \dots, \deg \varphi_{2s} \leq d$, we have
\begin{equation} \label{cht1}
 E_{s, \vec{\varphi}}(B) \ll_{s,d} |B|^{2s - \eta_s} \ \ \text{and} \ \ M_{s}(C) \ll_{s,d} |C|^{2s - \eta_s}.
 \end{equation}
Moreover, if $\varphi_i(0) \neq 0$ for each $1 \leq i \leq 2s$, then 
 \begin{equation} \label{cht2}
  M_{s,  \vec{\varphi}}(B) \ll_{s, d} |B|^{2s - \eta_s}. 
  \end{equation}
\end{theorem}

A first remark about Theorem $\ref{th3}$ is that it automatically recovers any sumset-product set type estimate of the form $\eqref{bclb}$. In particular, writing 
\[ X+ Y = \{ x+ y \ | \ x \in X, \ y \in Y\} \ \ \text{and} \  \ X\cdot Y= \{ x \cdot y \ | \ x \in X, \ y \in Y\} \]  
for any finite sets $X, Y\subseteq \mathbb{R},$ we see that since $\max \{|B|, |C|\} \geq |A|/2$ in the conclusion of Theorem $\ref{th3}$, we may apply Cauchy-Schwarz inequality, in a manner akin to $\eqref{csap}$, to obtain the following estimate on sumsets and product sets.

\begin{Corollary} \label{th1}
For any $d,s \in \mathbb{N}$ and any $\varphi_1, \dots, \varphi_s \in \mathbb{Q}[x]$ such that $1 \leq \deg \varphi_i \leq d$ for each $1 \leq i \leq s$ and for any finite set $A \subseteq \mathbb{Q}$, we have that
\[ |A^{(s)}| + |\varphi_1(A) + \dots + \varphi_s(A) | \gg_{s,d} |A|^{\eta_s} , \]
 for some $\eta_s \gg_{d} \log s / \log \log s$. Moreover if $\varphi_i(0) \neq 0$ for every $1 \leq i \leq s$, then
\[ |A^{(s)}| + |\varphi_1(A) \dots \varphi_s(A) | \gg_{s,d} |A|^{\eta_s}. \]
\end{Corollary}

Note that upon setting $\varphi_1, \dots, \varphi_s$ to be the same appropriately chosen linear polynomial in Corollary $\ref{th1}$, we are able to automatically recover the best known bound for Conjecture $\ref{esc}$ that is presented in \cite{PZ2020} as well as quantitatively strengthen the main result of \cite{HRZ2020}. In fact, Corollary $\ref{th1}$ presents the first set of bounds wherein one may choose a variety of different polynomials $\varphi_1, \dots, \varphi_s$ with unlike but bounded degrees and still obtain unbounded expansion. We further note that these so-called low-energy decompositions are much stronger than the corresponding sumset-product set type estimates. For instance, bounds of the form $\eqref{cht2}$ are new even in the situation when, say, $\varphi_1(x) = \dots = \varphi_{2s}(x) = c x + d$ for every $x \in \mathbb{Z}$, for some fixed $c,d \in \mathbb{Z}\setminus\{0\}$, while the corresponding sumset-product set type estimate was already proven, albeit with a quantitatively weaker exponent, in \cite{HRZ2020}.  Moreover, Theorem $\ref{th3}$ allows us to makes progress on some other problems with a sum-product flavour, which themselves generalise estimates of the form $\eqref{bclb}$. We describe one such result below.
\par



 Thus, given $\varphi \in \mathbb{Z}[x]$, we denote a finite set $X \subseteq \mathbb{Z}$ to be a $B_{s, {\varphi}}^+[1]$ set if for every $n \in \mathbb{Z}$, there is at most one distinct solution to the equation
\[ n = \varphi(x_1)+ \dots + \varphi(x_s ) ,\]
with $x_1, \dots, x_s \in X$. Here, we consider two such solutions to be the same if they
differ only in the ordering of the summands. Similarly, we denote $X$ to be a $B_{s, \varphi}^{\times}[1]$ set if for every $n \in \mathbb{Z}$, there is at most one distinct solution to the equation
\[ n = \varphi(x_1) \dots \varphi(x_s)  ,\]
with $x_1, \dots, x_s \in X$. When ${\varphi}$ equals the identity map, then a $B_{s, {\varphi}}^+[1]$ set is known as an \emph{additive Sidon set}, which we will denote as a $B_{s}^+[1]$ set. We define a $B_{s}^{\times}[1]$ set in an analogous fashion, and we refer to this as a \emph{multiplicative Sidon set}. These objects play a central role in the field of combinatorial number theory, and there has been a rich line of work investigating their various properties,  see \cite{Ci2001} and the references therein. While the main such inquiry often surrounds the size of the largest such set which is contained in an ambient set of integers, such as, say $\{1, 2, \dots, N\}$, this has recently been generalised to the setting of finding $B_{s}^+[1]$ and $B_s^{\times}[1]$ sets in arbitrary finite sets of integers, with the hope of having the size of at least one of them being much closer to the size of the ambient set, see \cite{JM2022, Po2021}. For instance, it is known that any finite set $A \subseteq \mathbb{Z}$ contains a $B_{s}^{+}[1]$ set of size at least $c_s' |A|^{1/s}$ for some $c_s' >0$ (see \cite{KSS1975},  \cite{Ru1995}), and this is sharp, up to multiplicative constants, by setting $A = \{1, \dots, N\}$. On the other hand,  $\{1, \dots, N\}$ contains large $B_{s}^{\times}[1]$ sets of size at least $N(\log N)^{-1}(1 - o(1))$, consider, for example, the set of primes up to $N$.
\par

It was shown in \cite{JM2022} that results akin to Theorem $\ref{th3}$ may be employed to make progress on such problems, and in this paper, we record some further improvements along this direction by proving the following result.

\begin{theorem} \label{sid1} 
For any $d, s \in \mathbb{N}$, there exists a parameter $\delta_s \gg_{d} \log s/ \log \log s$ such that the following holds true. Given any finite set $A \subseteq \mathbb{Z}$ and any ${\varphi}  \in \mathbb{Z}[x]$ such that $\deg \varphi = d$, the largest $B_{s, {\varphi}}^+[1]$ subset $X$ of $A$ and the largest $B_{s}^{\times}[1]$ subset $Y$ of $A$ satisfy
\[ \max\{|X|,|Y|\} \gg_{s,d} |A|^{\delta_s/s}. \]
Moreover if $\varphi(0) \neq 0$, then writing $X'$ to be the largest $B_{s, \varphi}^{\times}[1]$ subset of $A$, we have
\[ \max\{|X'|, |Y|\} \gg_{s,d} |A|^{\delta_s/s}. \]
\end{theorem}

The first such result was proven in \cite{JM2022}, where, amongst other estimates, it was shown that any finite set $A$ of integers contains either a $B_{s}^{+}[1]$ set or a $B_{s}^{\times}[1]$ set of size at least $|A|^{\delta_s/s}$, with $\delta_s \gg (\log \log s)^{1/2 - o(1)}$. This was subsequently improved by Shkredov \cite{Shk2022} who gave the bound $\delta_s \gg (\log s)^{1/2 - o(1)}$. Theorem $\ref{sid1}$ not only quantitatively improves this to $\delta_s \gg (\log s)^{1 - o(1)}$, but it also generalises this to estimates on $B_{s, {\varphi}}^{+}[1]$ and $B_{s, {\varphi}}^{\times}[1]$ sets, for suitable choices of ${\varphi} \in \mathbb{Z}[x]$. Moreover, this can be seen as an alternate generalisation of the sum-product phenomenon, since the first conclusion of Theorem $\ref{sid1}$ implies that
\[ |s \varphi(A)| + |A^{(s)}| \geq |s \varphi(X)| + |Y^{(s)}| \gg_{s} |X|^s + |Y|^s \gg_{s} |A|^{(\log s)^{1 - o(1)}}. \]
As before, this recovers the best known result towards Conjecture $\ref{esc}$ for large values of $s$ by simply setting $\varphi$ to be the identity map.
\par

In consideration of Conjecture $\ref{esc}$, one may naively expect Theorems $\ref{th3}$ and $\ref{sid1}$ to hold for every $\eta_s < s$ and $\delta_s <s$ respectively, since either of these would imply Conjecture $\ref{esc}$ in a straightforward manner, but in fact, both of these have been shown to be false in various works. In \cite{BW2017}, Balog--Wooley constructed arbitrarily large sets $A$ of integers, such that for every $B \subseteq A$ with $|B| \geq |A|/2$, we have that
\[ E_{s}(B) , M_{s}(B) \gg_{s} |A|^{s + (s-1)/3}, \]
thus implying that $\eta_s \leq (2s+1)/3$. Similarly, in the case when $\varphi$ is set to be the identity map, it was noted by Roche-Newton that these sets also gave the bound $\delta_2 \leq 3/2$, thus refuting a question of Klurman--Pohoata\cite{Po2021}. This was then improved by constructions of Green--Peluse (unpublished), Roche-Newton--Warren \cite{RNW2021} and Shkredov \cite{Sh2021} to $\delta_2 \leq 4/3$. The latter constructions were then generalised in \cite[Proposition 1.5]{JM2022} for every $s \in \mathbb{N}$, thus giving the estimate $\delta_s/s \leq 1/2 + 1/(2s+2)$ when $s$ is even and $\delta_s/s \leq 1/2 + 1/(2s)$ when $s$ is odd.
\par

We now return to the even more general case where $s,d$ are some natural numbers and $\varphi \in \mathbb{Z}[x]$ is some arbitrary polynomial with $\deg \varphi = d$. In this setting, we are able to utilise ideas from \cite{BW2017} and \cite{JM2022} to record the following upper bounds. 

\begin{Proposition} \label{bwex}
Let $d,s, N$ be natural numbers and let $\varphi \in (\mathbb{Z}[x])^{2s}$ satisfy $\vec{\varphi} = (\varphi, \dots, \varphi)$ for some $\varphi$ with $\deg \varphi = d$. If $s \geq 10d(d+1)$, then there exists a finite set $A \subseteq \mathbb{N}$ satisfying $N \ll |A|$ such that for any $B \subseteq A$ with $|B| \geq |A|/2$, we have that
\[ E_{s, \vec{\varphi}}(B) , M_{s}(B) \gg_{s, \vec{\varphi}} |A|^{s + s/3 - (d^2 + d + 2)/6} . \]
Similarly, when $s$ is even, there exists a finite set $A' \subseteq \mathbb{N}$ with $N \ll |A'|$ such that the largest $B_{s, {\varphi}}^+[1]$ subset $X$ and the largest $B_{s}^{\times}[1]$ subset $Y$ of $A'$ satisfy
\[ |X| \ll_{s,d} |A'|^{d/s} (\log |A'|)^{2d/s} \ \ \text{and} \ \ |Y| \ll_{s} |A'|^{1/2 + 1/(2s + 2)}.  \]
\end{Proposition}

This, along with Theorems $\ref{th3}$ and $\ref{sid1}$, implies that whenever $s \geq 10d(d+1)$, then
\[ (\log s)^{1 - o(1)} \ll_{d} \eta_s \leq 2s/3 + (d^2 + d + 2)/6 \ \ \text{and} \  \ (\log s)^{1 - o(1)} \ll_{d} \delta_s \leq s/2 + 1/2 \]
holds true. Thus, there is a large gap between the known upper and lower bounds, and it would be interesting to understand the right order of magnitude for these quantities, see also \cite[Question $1.6$]{JM2022}. We remark that Proposition $\ref{rev4}$ provides further relations between Theorems $\ref{th3}$ and $\ref{sid1}$.
\par

Returning to our original theme, we note that a combination of Theorem $\ref{th3}$ and the Pl{\"u}nnecke--Ruzsa inequality (see Lemma \ref{pr21}) implies that whenever a set $A \subseteq \mathbb{Z}$ satisfies $|A^{(2)}| \leq K|A|$ with $K = |A|^{\eta_s/10s}$, then the sumset $\varphi_1(A) + \dots + \varphi_s(A)$ exhibits unbounded expansion. But in fact, in this case, we are able to show that such sumsets must be close to being extremally large. This can be deduced from the following more general result. As is usual, we write $e(\theta) = e^{2 \pi i \theta}$ for every $\theta \in \mathbb{R}$. 

\begin{theorem} \label{mve}
Let $K \geq 1$ be a real number and let $d,s$ be natural numbers. Moreover, let $A \subseteq \mathbb{Z}$ be a finite set such that $|A^{(2)}| = K|A|$ and let $\varphi \in \mathbb{Z}[x]$ satisfy $\deg \varphi = d$ and let $\frak{a}: \mathbb{Z} \to [0, \infty)$ be a function. Then, writing $C = 8 + 12 \log (d^2 + 2) + 6 \log (2s)$, we have
\begin{equation} \label{ze3}
 \int_{[0,1)} | \sum_{a \in A} \frak{a}(a) e( \alpha \varphi(a) ) |^{2s} d \alpha \ll_{s} K^{Cs} ( \log 2|A|)^{2s} \Big(\sum_{a \in A} \frak{a}(a)^2 \Big)^s . 
\end{equation}
\end{theorem}

When $\varphi(x) = x$ for every $x \in \mathbb{Z}$,  such types of results have been previously referred to as the weak Erd\H{o}s-Szemer\'{e}di Conjecture,  which suggests that whenever $|A^{(2)}| \leq K|A|$, then we have $|2A| \gg_{\varepsilon} |A|^{2 - \varepsilon}/ K^{O(1)}$ for every $\varepsilon>0$.  Proving such an estimate formed a key step in the work of Bourgain--Chang \cite{BC2004},  and in fact,  Theorem \ref{mve} improves upon \cite[Theorem $1.3$]{PZ2020} by replacing small powers of $|A|$ with a factor of $(\log |A|)^{O(1)}$. 
\par

Denoting the mean value on the left hand side of $\eqref{ze3}$ to be $E_{s, \frak{a}, \varphi}(A)$, we may apply orthogonality to see that
\[ E_{s, \frak{a}, \varphi}(A)  = \sum_{a_1, \dots, a_{2s} \in A} \frak{a}(a_1) \dots \frak{a}(a_{2s})\mathds{1}_{\varphi(a_1) + \dots \varphi(a_s) = \varphi(a_{s+1}) + \dots + \varphi(a_{2s})} \geq \Big(\sum_{a \in A} \frak{a}(a)^2\Big)^s , \]
where the inequality follows from counting the diagonal solutions $a_{i} =  a_{i+s}$ for all $1 \leq i \leq s$. On the other hand, given $\varepsilon >0$, in the case when $|A^{(2)}| \ll_{s,d} |A|^{ c\varepsilon}|A|$ for some small constant $c = c(d,s)>0$, we may apply Theorem $\ref{mve}$ to deduce that
\[ E_{s, \frak{a}, \varphi}(A)  \ll_{s,d, \varepsilon} |A|^{\varepsilon} \Big(\sum_{a \in A} \frak{a}(a)^2\Big)^s , \]
which matches the aforementioned lower bound up to a factor of $|A|^{\varepsilon}$. Thus, Theorem $\ref{mve}$ indicates that additive polynomial equations exhibit strongly diagonal type behaviour with respect to multiplicatively structured sets of integers. In fact, our methods can prove a corresponding result for multiplicative polynomial equations, which, in turn, can be employed to furnish a non-linear analogue of a subspace-type theorem, see Theorems $\ref{fin6}$ and $\ref{mltsz}$.
\par

We remark that apart from the sum-product conjecture,  our results are motivated by another well-known phenomenon which studies growth of sets of the form $F(A, \dots, A) = \{ F(a_1, \dots, a_s) \ | \ a_1, \dots, a_s \in A\}$, where $F$ is an arbitrary polynomial in $s$ variables.  Estimates for $|F(A, \dots, A)|$,  as well as its connection to the size of the product set $A^{(s)}$, have been widely studied in previous works,  see,  for instance, \cite{BT2012,  RSS2016,  RS2020, Ta2015}. In most of such papers, the motivation is to find some $0<c\leq1$ such that either $|F(A, \dots, A)|\gg |A|^{1+ c}$ holds for all large,  finite subsets $A$ of some ambient set,  or $F$ is of a special form,  such as,  say $F(x_1,  \dots,  x_s) = \varphi_1(x_1) + \dots + \varphi_s(x_s)$ or $F(x_1, \dots, x_s) = \varphi_1(x_1) \dots \varphi_s(x_s),$ where $\varphi_1, \dots, \varphi_s$ are univariate polynomials.  On the other hand,  our results show that over $\mathbb{Q}$,  if $F$ is of the latter type with suitably chosen $\varphi_1, \dots, \varphi_s$,  then either $|F(A,  \dots, A)|$ or $|A^{(s)}|$  exhibits unbounded expansion.  Thus our results compliment the regimes analysed by these previous works.  Moreover,  the methods involved in the latter seem to be quite different from the techniques used in our paper.
\par

We now proceed to describe the outline of our paper, along with some of the proof ideas present therein. As previously mentioned, \S2 is dedicated to presenting some further applications of our method. In \S3, we record some properties of the mixed energies $E_{s, \vec{\varphi}}(A)$ and $M_{s, \vec{\varphi}}(A)$ that we will use throughout our paper. The first main step towards proving Theorem $\ref{th3}$ is initiated in \S4, which we utilise to prove a generalisation of a result of Chang \cite[Proposition 8]{Ch2003}. This is the content of Lemma $\ref{chang}$, which can be interpreted as a decoupling type inequality. For instance, in the additive case, suppose that we have natural numbers $d,r,s$, a prime number $p$, a polynomial $\varphi \in \mathbb{Z}[x]$ with $\deg \varphi = d$, and finite sets $A_0, A_1, \dots, A_r$ of natural numbers such that for every $1 \leq i \leq r$, we have $A_i = \{ a \in A_0 \ | \ \nu_p(a) = n_i \}$ for some unique $n_i \in \mathbb{N} \cup \{0\}$, where $\nu_p(n)$ denotes the largest exponent $m \in \mathbb{N} \cup \{0\}$ such that $p^m$ divides $n$. Writing
\[ f_i(\alpha) = \sum_{a \in A_i}\frak{a}(a) e( \alpha \varphi(a)) \ \text{for all} \ \alpha \in [0,1) \ \ \text{and} \ \ \n{ f_i}_{2s} = \Big(\int_{[0,1)} |f(\alpha)|^{2s} d \alpha \Big)^{1/2s},  \]
where $\frak{a}: \mathbb{N} \to [0, \infty)$ is some function, we are interested in proving estimates of the form
\[   \n{f_1 + \dots + f_r}_{2s} \ll_{d,s} \Big(\sum_{i=1}^{r} \n{f_i}_{2s}^2 \Big)^{1/2},\]
that is, we want to exhibit square-root cancellation in moments of these exponential sums. This may then be iterated to obtain estimates on $E_{s, \frak{a}, \varphi}(A)$ in terms of the $l^2$-norm of $\frak{a}$ and the so-called \emph{query-complexity} $q(A)$ of $A$, see Lemma $\ref{ayay}$ for more details.
\par

We now proceed to \S5, where we analyse the multiplicative analogue of this phenomenon. This ends up being harder to deal with, and in particular, we have to first study some auxilliary mean values of the form 
\begin{equation} \label{defdec}
J_{s, \frak{a}, \varphi}(A) =   \sum_{a_1, \dots, a_{2s} \in A} \frak{a}(a_1) \dots \frak{a}(a_{2s}) \mathds{1}_{a_1 \dots a_s = a_{s+1} \dots a_{2s}} \mathds{1}_{\varphi(a_1) \dots \varphi(a_{s}) = \varphi(a_{s+1})\dots \varphi(a_{2s})}, \end{equation}
and prove decoupling type inequalities for such quantities. Iterating these estimates, as in \S4, leads to suitable bounds for $J_{s, \frak{a}, \varphi}(A)$. Next, we prove a multiplicative variant of an averaging argument from analytic number theory, which allows us to discern bounds on $M_{s, \frak{a}, \varphi}(A)$ from estimates for $J_{s, \frak{a}, \varphi}(A)$ by incurring a further factor of $|A^{(s)}/A^{(s)}|$, see Lemmata $\ref{tst}$ and $\ref{trut}$. 
\par

At this stage of the proof, we collect various inverse theorems from arithmetic combinatorics in \S6, the first of these arising from our work on a variant of the $s$-fold Balog-Szemer\'{e}di-Gowers theorem in \cite{Mu2021c}. This, along with a dyadic pigeonholing trick, allows us to deduce that whenever $A$ satisfies $M_{s}(A) \geq |A|^{2s- k}$, for suitable values of $s,k$, then $A$ has a large intersection with a set $U'$ satisfying $|U'| \ll |A|^k$, such that the many-fold product sets of $U'$ expand slowly, see Theorem $\ref{th46}$ for more details. The other inverse theorem that we are interested in emanates from the circle of ideas recorded in \cite{PZ2020}, and it implies that given two finite sets $A, X \subseteq \mathbb{N}$ satisfying the inequality $|A \cdot X \cdot X| \leq K |X|$, there must exist a large subset $B$ of $A$ with $q(B) \leq \log K$. Roughly speaking, this can be seen as a Freiman type structure theorem for sets with small asymmetric product sets, since query-complexity itself may be interpreted as a skewed version of some notion of multiplicative dimension. The two aforementioned inverse theorems combine naturally to imply that any set $A \subseteq \mathbb{Z}$ with a large multiplicative energy has a large subset $B$ with a small query complexity. In \S7 we apply this idea iteratively, along with the results recorded in \S\S3-6, to yield the proof of Theorem $\ref{th3}$. We utilise \S8 to prove Theorems $\ref{mve}$ and $\ref{fin6}$, and finally, in \S9, we provide the proofs of Theorem $\ref{sid1}$ and Proposition $\ref{bwex}$.
\par


\textbf{Notation}. In this paper, we use Vinogradov notation, that is, we write $X \gg_{z} Y$, or equivalently $Y \ll_{z} X$, to mean $X \geq C_{z} |Y|$ where $C$ is some positive constant depending on the parameter $z$.  We use $e(\theta)$ to denote $e^{2\pi i \theta}$ for every $\theta \in \mathbb{R}$. Moreover, for every natural number $k \geq 2$ and for every non-empty, finite set $Z$, we use $|Z|$ to denote the cardinality of $Z$, we write $Z^k = \{ (z_1, \dots, z_k)  \ |  \ z_1, \dots, z_k \in Z\}$ and we use boldface to denote vectors $\vec{z} = (z_1, z_2, \dots, z_k) \in Z^k$. All our logarithms will be with respect to base $2$.

\textbf{Acknowledgements}. The author is supported by Ben Green's Simons Investigator Grant, ID 376201. The author is grateful to Ben Green and Oliver Roche-Newton for helpful discussions.  The author would like to thank the anonymous referee for a careful reading of the manuscript and for various helpful comments.


\section{Further applications}


As mentioned in the previous section, we are able to prove various multiplicative analogues of Theorem $\ref{mve}$. In order to state these, we first define a generalisation of $M_{s, \vec{\varphi}}(A)$ and $J_{s, \frak{a}, \varphi}(A)$, and so, for every $\vec{\varphi} \in \mathbb{Q}[x]^{2s}$, every $\frak{a}: \mathbb{R} \to [0, \infty)$ and every finite set $A \subseteq \mathbb{R}$, we define
\[ M_{s, \frak{a}, \vec{\varphi}}(A) = \sum_{a_1, \dots, a_{2s} \in A} \frak{a}(a_1) \dots \frak{a}(a_{2s}) \mathds{1}_{\varphi_1(a_1) \dots \varphi_{s}(a_s) = \varphi_{s+1}(a_{s+1}) \dots \varphi_{2s}(a_{2s})} \]
\text{and}
\[J_{s, \frak{a}, \vec{\varphi}}(A) =   \sum_{a_1, \dots, a_{2s} \in A} \frak{a}(a_1) \dots \frak{a}(a_{2s}) \mathds{1}_{a_1 \dots a_s = a_{s+1} \dots a_{2s}} \mathds{1}_{\varphi_1(a_1) \dots \varphi_s(a_{s}) = \varphi_{s+1}(a_{s+1})\dots \varphi_{2s}(a_{2s})}. \]
Moreover, for any $\varphi \in \mathbb{R}[x]$, we write $\mathcal{Z}_{\varphi} = \{ x \in \mathbb{R} \ | \ \varphi(x) = 0\}$. Similarly, for any $\vec{\varphi} = (\varphi_1, \dots, \varphi_{2s}) \in (\mathbb{R}[x])^{2s}$, we denote $\mathcal{Z}_{\vec{\varphi}} = \mathcal{Z}_{\varphi_1} \cup \dots \cup \mathcal{Z}_{\varphi_{2s}}$. Using methods related to the proofs of Theorems $\ref{th3}$ and $\ref{mve}$, we are able to derive the following upper bounds for $J_{s, \frak{a}, \vec{\varphi}}(A)$ and $M_{s, \frak{a}, \vec{\varphi}}(A)$, whenever $A \cdot A$ is small.

\begin{theorem} \label{fin6}
Let $K \geq 1$ be a real number and let $d,s$ be natural numbers. Moreover,  let $\vec{\varphi} \in \mathbb{Q}[x]^{2s}$ satisfy $1 \leq \deg \varphi_i \leq d$ for each $1 \leq i \leq 2s$, let $A$ be a finite subset of $\mathbb{Q} \setminus( \mathcal{Z}_{\vec{\varphi}}\cup \{0\})$ such that $|A \cdot A| = K|A|$, and let $\frak{a}: \mathbb{N} \to [0, \infty)$ be a function. Then 
\[ \max \{ J_{s, \frak{a},  \vec{\varphi}}(A)) ,  M_{s, \frak{a}, \vec{\varphi}}(A)|A|^{-1} \}  \ll_{s,d} K^{Cs} (\log |A|)^{2s} (\sum_{a \in A} \frak{a}(a)^2)^{s}   , \]
where $C = 10 + 48\log (d+3) + 6 \log (2s)$.
\end{theorem}

Here, the condition $A \subseteq \mathbb{Q}\setminus( \mathcal{Z}_{\vec{\varphi}}\cup \{0\})$ seems to be necessary. In order to see this, we may choose $\varphi \in \mathbb{Q}[x]$ to be some linear polynomial, $\vec{\varphi} = (\varphi, \dots, \varphi)$ and $A = \{2,4,\dots, 2^N\} \cup \{x\}$ for some $x \in \mathcal{Z}_{\varphi}$. In this case, since $\varphi(a_1) \dots \varphi(a_{s-1}) \varphi(x) = \varphi(a_{s+1}) \dots \varphi(a_{2s-1}) \varphi(x) = 0$ for any $a_1, \dots, a_{s-1}, a_{s+1}, \dots, a_{2s-1} \in A$, we see that $J_{s, \vec{\varphi}}(A) \geq M_{s-1}(A) \gg_s N^{2s - 3}$.
\par

We will now use the above result to prove a non-linear analogue of a subspace-type theorem. In particular, it was noted in \cite{HRZ2020} that a quantitative version of the well-known subspace theorem of Evertse, Schmidt and Schlikewei \cite{ESS2002} combined together with Freiman's lemma \cite[Lemma 5.13]{TV2006} implies that for any fixed $c_1, c_2 \in \mathbb{C} \setminus \{0\}$ and for any finite subset $A \subseteq \mathbb{Q}$ with $|A \cdot A| = K|A|$, we have
\[ \sum_{a_1, a_2 \in A} \mathds{1}_{c_1 a_1 + c_2 a_2 = 1} \leq (16)^{2^6(2K + 3)}. \]
In fact,  a more general result can be proven via these techniques which holds for linear equations with many variables (see \cite[Corollary $1.6$]{HRZ2020}),  but for simplicity of exposition, we restrict to the two-fold case here.  While the above inequality is very effective for small values of $K$; it may deliver worse than trivial bounds when $K$ is large. For instance, in the case when $K > c \log |A|$ for some constant $c > (4\log 2)^{-1}$, the right hand size becomes much larger than the trivial upper bound $|A|$. The authors of \cite{HRZ2020} asked whether the above upper bound could be improved to have a polynomial dependence in $K$, and proved that 
 for any fixed $c_1, c_2 \in \mathbb{Q} \setminus \{0\}$ and for any $\varepsilon >0$ and $A \subseteq \mathbb{Q}$ with $|A \cdot A| = K|A|$,  one has
\[ \sum_{a_1, a_2 \in A} \mathds{1}_{c_1 a_1 + c_2 a_2 = 1} \ll_{\varepsilon} K^{C_{\epsilon}}|A|^{\varepsilon}, \]
for some constant $C_{\varepsilon} >0$. Using Theorem $\ref{fin6}$, we can prove a non-linear analogue of the above result.

\begin{theorem} \label{mltsz}
Let $A$ be a finite subset of $\mathbb{Q}$ such that $|A \cdot A| \leq K|A|$ for some $K \geq 1$, let $\varphi \in \mathbb{Q}[x]$ have $\deg \varphi = d \geq 1$ with $\varphi(0) \neq 0$ and let $\varepsilon >0$. Then 
\[ \sum_{a_1, a_2 \in A} \mathds{1}_{a_1 = \varphi(a_2) } \ll_{d, \epsilon} K^{C} |A|^{\varepsilon},\]
for some constant $C = C(d,\varepsilon) >0$.
\end{theorem}

\begin{proof}
We begin by partitioning $\mathbb{R}$ as $\mathbb{R} = I_1 \cup \dots \cup I_{r} \cup I_{r+1}$, where $r \ll_d 1$ and $I_1, \dots, I_r$ are open intervals such that the derivative $\varphi'$ is non-zero on each such interval and such that the set $Z_{\varphi} \cup \{0\} \subseteq I_{r+1}$ with $|I_{r+1}| \ll_{d} 1$. Let $A_i = A \cap I_{i}$ for each $1 \leq i \leq r+1$. Next, given any two finite sets $X,Y$, we denote $r(X,Y)$ to be the number of solutions to the equation $x = \varphi(y)$, with $x \in X,  y \in Y$.  With this notation in hand, we note that for any $1 \leq i \leq r$, we have that $r(A_{r+1}, A_i) + r(A_i, A_{r+1}) \ll_d |A_{r+1}| \leq |I_{r+1}| \ll_d 1$ since fixing either of $x,y$ in the equation $x = \varphi(y)$ fixes the other variable up to $O_d(1)$ choices.   Thus, it suffices to upper bound $r(A_i, A_j)$ for every $1 \leq i, j \leq r$.  Fixing some $1 \leq i,j \leq r$, we define $S = \{ a \in A_j : \varphi(a) \in A_i\}$, whence,  $r(A_i,A_j) \ll_d |S|$. Since $S \subseteq A$, we may utilise Theorem $\ref{fin6}$ to deduce that for every $s \geq 2$, we have
\[ M_{s, \varphi}(S)  \ll_{s,d} |A| K^{ Cs} (\log |A|)^{2s} |S|^{s} , \]
where $C = 10 + 24 \log (d+3) + 6 \log (2s)$. Applying Cauchy-Schwarz inequality, we get that
\[ |\varphi(S)^{(s)}| \gg_{s,d} |S|^{s} K^{-Cs} (\log |A|)^{-2s} |A|^{-1}. \]
On the other hand, since $\varphi(S) \subseteq A_i \subseteq A$, we may apply the Pl{\"u}nnecke--Ruzsa  inequality (see Lemma \ref{pr21}) to deduce that
\[ |\varphi(S)^{(s)}|  \leq |A^{(s)}| \leq K^{s} |A|. \]
This, together with the preceding expression, implies that
\[ |S| \ll_{s,d} K^{C+1} (\log |A|)^2 |A|^{2/s},\]
for every $s \geq 2$. Noting the fact that $\log |A| \ll_{\varepsilon} |A|^{\varepsilon/4}$, we may choose $s$ to be sufficiently large in terms of $\varepsilon$, say, $s = \ceil{10 \epsilon^{-1}}$ to obtain the desired result.
\end{proof}

Theorem $\ref{mltsz}$ can be interpreted as a bound on the number of points of $A \times A$ that lie on the polynomial curve $y = \varphi(x)$. Naturally, one may employ this result to further prove an incidence estimate for point sets that are of the form $A \times A$, with $A \subseteq \mathbb{Q}$ satisfying $|A \cdot A| \leq K|A|$ for some parameter $K \geq 1$, and finite sets of curves of the form $y = \varphi(x)$, where $\varphi \in \mathbb{Q}[x]$ satisfies $\varphi(0) \neq 0$, but we do not pursue this here.
\par



\section{Properties of mixed energies}

We begin this section by recording some notation. For every real number $\lambda \neq 0$ and every $\varphi \in \mathbb{R}[x]$, we define the polynomial $\varphi_{\lambda} \in \mathbb{R}[x]$ by writing $\varphi_{\lambda}(x) = \varphi (\lambda x)$ for every $x \in \mathbb{R}$. Furthermore, given $s \in \mathbb{N}$ and $\vec{\varphi} \in \mathbb{R}[x]^{s}$, we denote $\vec{\varphi}_{\lambda} = (\varphi_{1, \lambda}, \dots, \varphi_{s, \lambda})$ and we write $\lambda \cdot \vec{\varphi} = (\lambda \varphi_1, \dots, \lambda \varphi_{s})$. It is worth noting that 
\[  \mathcal{Z}_{\vec{\varphi}_{\lambda}} = \lambda^{-1} \cdot \mathcal{Z}_{\vec{\varphi}} \ \ \text{and} \ \  \mathcal{Z}_{\lambda \cdot \vec{\varphi}} = \mathcal{Z}_{\vec{\varphi}}, \]
where for every finite set $X \subseteq \mathbb{R}$ and for every $\nu \in \mathbb{R}$, we denote $\nu \cdot X = \{ \nu x \ | \ x \in X\}$. Next, given some function $\frak{a} : \mathbb{R} \to [0, \infty)$ and some finite sets $A_1, \dots, A_{2s} \subseteq \mathbb{R}$, we define
\[ E_{s, \frak{a}}(A_1, \dots A_{2s}) = \sum_{a_1, \dots, a_{2s} \in A} \frak{a}(a_1) \dots \frak{a}(a_{2s}) \mathds{1}_{a_1 + \dots + a_{s} = a_{s+1} + \dots +a_{2s}} . \]
If $A_1 = \dots = A_{2s} = A$, we denote $E_{s, \frak{a}}(A) = E_{s, \frak{a}}(A_1, \dots A_{2s})$. Moreover, when the function $\frak{a}$ satisfies $\frak{a}(x) = 1$ for every $x \in \mathbb{R}$, we suppress the dependence on $\frak{a}$, and thus, we write $E_{s}(A_1, \dots, A_{2s}) = E_{s, \frak{a}}(A_1, \dots, A_{2s})$.
\par 

Our first aim in this section is to prove the following generalisation of \cite[Lemma $3.2$]{Mu2021c}.

\begin{lemma} \label{reprove}
Let $A_1, \dots, A_{2s}$ be a finite sets of real numbers and let $\frak{a} : \mathbb{R} \to [0, \infty)$. Then
\[ E_{s,\frak{a}}(A_1, \dots, A_{2s}) \leq E_{s, \frak{a}}(A_1)^{1/2s} \dots E_{s, \frak{a}}(A_{2s})^{1/2s}. \]
Moreover, for natural number $r$ and for finite subsets $A_1, \dots, A_r \subseteq \mathbb{R}$, we have
\[ E_{s, \frak{a}}(A_1 \cup \dots \cup A_r) \leq r^{2s} \sup_{1 \leq i \leq r} E_{s, \frak{a}}(A_i) . \]
Finally, if $\frak{a}(x) = 1$ for every $x \in \mathbb{R}$, then for every $1 \leq l < s$, we have
\[ E_{s, \frak{a}}(A_1) \leq |A_1|^{2s - 2l} E_{l, \frak{a}}(A_1) . \]
\end{lemma}

\begin{proof}
We begin by focusing on proving the first inequality.  Firstly,  note that this inequality remains invariant under replacing the function $\frak{a}$ by $\frak{a}/M$,  for any $M > 0$.  Moreover,  since $A_1,\dots, A_{2s}$ are finite sets,  we may set $M = \max_{a \in A_1 \cup \dots \cup A_{2s}} \frak{a}(a) + 100$ in the preceding statement to ensure that $0 \leq \frak{a}(a) \leq 1$ for any $a \in A_1 \cup \dots \cup A_{2s}$.  We further point out that it suffices to prove that for every $\varepsilon >0$, we have
\[ E_{s, \frak{a}}(A_1, \dots, A_{2s}) \leq E_{s, \frak{a}}(A_1)^{1/2s} \dots E_{s, \frak{a}}(A_{2s})^{1/2s} + \varepsilon. \]
For the purposes of this proof, we define, for each $(\xi, R) \in \mathbb{R}^2$ satisfying $\xi \neq 0$ and $R>0$, the quantity $I(R, \xi) = \int_{[0,R]} e( \xi \alpha) d \alpha$. When $\xi \neq 0$, we see that $|I(R, \xi)| \ll |\xi |^{-1}$ while $I(R,0) = R$. 
We now define, for each $1 \leq i \leq 2s$, the exponential sum $f_i : [0, \infty) \to \mathbb{C}$ as 
\[ f_i(\alpha) = \sum_{a \in A_i} \frak{a}(a) e(a \alpha) , \]
and we let $X = A_1 \cup \dots \cup A_{2s}$. Finally, we write
\[ \xi_0 = \min_{\vec{a} \in X^{2s} \ \text{such that} \ a_1 + \dots - a_{2s} \neq 0} |a_1 + \dots - a_{2s} | . \]
\par

With this discussion in hand, it is straightforward to note that for each $R>0$, we have
\[ \int_{[0,R]} f_1(\alpha) \dots \overline{f_{2s}(\alpha)} d \alpha = \sum_{\vec{a} \in A_1 \times \dots \times  A_{2s} } \frak{a}(a_1) \dots \frak{a}(a_{2s}) I(R, a_1 + \dots - a_{2s}), \]
whence, 
\[ \int_{[0,R]} f_1(\alpha) \dots \overline{f_{2s}(\alpha)} d \alpha  = R E_{s, \frak{a}}(A_1, \dots, A_{2s}) + O(|A_1| \dots |A_{2s}| \xi_0^{-1} ) ,\]
where we have used the fact that  $0 \leq \frak{a}(a) \leq 1$ for any $a \in A_1 \cup \dots \cup A_{2s}$.  Similarly, for each $1 \leq i \leq 2s$, we have that
\begin{equation} \label{sof2}
 \int_{[0, R]} |f_i(\alpha)|^{2s} d\alpha = R E_{s, \frak{a}}(A_i) + O(|A_i|^{2s} \xi_0^{-1}). 
\end{equation}
Amalgamating these expressions with a standard application of H\"{o}lder's inequality gives us
\[ E_{s, \frak{a}}(A_1, \dots, A_{2s}) + O( R^{-1} |A_1|\dots |A_{2s}|  \xi_0^{-1} ) \leq \prod_{i=1}^{2s} (E_{s, \frak{a}}(A_i) + O(R^{-1} |A_i|^{2s} \xi_0^{-1} ) )^{1/2s} . \]
Choosing $R$ to be sufficiently large, say $R \geq \varepsilon^{-1} (4s^2 |A_1| \dots |A_{2s} |)^{4s^2} \xi_0^{-1}$, delivers the first inequality stated in our lemma.
\par
The second inequality can be swiftly deduced from the first inequality since
\begin{align*}
 E_{s, \frak{a}}(A_1 \cup \dots \cup A_r)  
 & = \sum_{1 \leq  i_1, \dots, i_{2s} \leq r}E_{s, \frak{a} }(A_{i_1}, \dots, A_{i_{2s}}) 
 \leq \sum_{1 \leq  i_1, \dots, i_{2s} \leq r} \prod_{j=1}^{2s}   E_{s, \frak{a}}(A_{i_j})^{1/2s} \\
 & = \prod_{j=1}^{2s} \Big(\sum_{i=1}^{r} E_{s, \frak{a}}(A_i)^{1/2s}\Big) \leq r^{2s} \sup_{1 \leq i \leq r} E_{s, \frak{a}}(A_i).
\end{align*}
The proof of the third inequality follows similarly, wherein, we see that it suffices to show
\[ E_{s, \frak{a}}(A_1) \leq |A_1|^{2s - 2l} E_{l, \frak{a}}(A_1) + \varepsilon,  \]
for each $\varepsilon >0$. This follows from noting the fact that $|f_1(\alpha)| \leq |A|$ for each $0 \leq \alpha \leq R$ along with $\eqref{sof2}$, and then choosing $R$ to be some sufficiently large real number.
\end{proof}

We now generalise Lemma $\ref{reprove}$ for mixed energies of the form $J_{s, \frak{a}, \vec{\varphi}}(A)$. Thus,  let $A_1, \dots, A_{2s}$ be finite subsets of $\mathbb{Z}$, let $\vec{\varphi} \in \mathbb{Z}[x]^{2s}$ be a vector and let $\frak{a}:\mathbb{N} \to [0, \infty)$ be a function supported on $A_1 \cup \dots \cup A_{2s}$. We define $J_{s, \frak{a}, \vec{\varphi}}(A_1, \dots, A_{2s})$ to be the quantity
\[  \sum_{a_1 \in A_1} \dots \sum_{a_{2s} \in A_{2s}} \frak{a}(a_1) \dots \frak{a}(a_{2s}) \mathds{1}_{a_1 \dots a_s = a_{s+1} \dots a_{2s}} \mathds{1}_{\varphi_1(a_1) \dots \varphi_{s}(a_s) = \varphi_{s+1}(a_{s+1}) \dots \varphi_{2s}(a_{2s})}, \]
that is, a weighted count of the number of solutions to the system of equations
\begin{equation} \label{mixedsys}
 x_1 \dots x_s = x_{s+1} \dots x_{2s} \ \ \text{and} \ \ \varphi_1(x_1) \dots \varphi_s(x_s) = \varphi_{s+1}(x_{s+1}) \dots \varphi_{2s}(x_{2s}) ,
 \end{equation}
with $x_i \in A_i$ for each $1 \leq i \leq 2s$. We now present our second lemma that we will prove in this section.

\begin{Proposition} \label{gvup}
Let $\vec{\varphi} \in \mathbb{Q}[x]^{2s}$ be a vector such that $\deg \varphi_i \leq d$ for every $1 \leq i \leq 2s$ and let  $A_1, \dots, A_{2s}$ be finite subsets of $\mathbb{Q}\setminus (\{0\} \cup \mathcal{Z}_{\vec{\varphi}})$, and let $\frak{a}: \mathbb{Q} \to [0,\infty)$ be a function. Then we have that
\begin{equation} \label{eqz1} J_{s, \frak{a}, \vec{\varphi}}(A_1, \dots, A_{2s}) \leq (d+2)^{2s} J_{s, \frak{a}, \varphi_1}(A_1)^{1/2s} \dots J_{s, \frak{a}, {\varphi}_{2s}}(A_{2s})^{1/2s}   .    
\end{equation}
Moreover,  if for each $1 \leq j \leq 2s$,  we have some open interval $I_j \subseteq \mathbb{R}$ such that $x y >0$ and $\varphi_i(x) \varphi_i(y) > 0$ for every $x,y \in I_j$ and $A_j \subseteq I_j$,  then
\begin{equation} \label{eqz2}
 J_{s, \frak{a}, \vec{\varphi}}(A_1, \dots, A_{2s}) \leq J_{s, \frak{a}, \varphi_1}(A_1)^{1/2s} \dots J_{s, \frak{a}, {\varphi}_{2s}}(A_{2s})^{1/2s}   .    
\end{equation}
\end{Proposition}

\begin{proof}
We begin by noting that 
\[ \prod_{i=1}^{s} \varphi_i(x_i)  =  \prod_{i=1}^{s} \varphi_{s+i}(x_{s+i})  \ \ \text{and} \ \ \prod_{i=1}^{s} x_i = \prod_{i=1}^{s} x_{s+i} \]
holds true if and only if we have
\[ \prod_{i=1}^{s}  {\lambda}^{d+1}  \varphi_{i, {\lambda}^{-1}}({\lambda} x_i) = \prod_{i=1}^{s}  {\lambda}^{d+1}  \varphi_{s+i, {\lambda}^{-1}}({\lambda} x_{s+i}) \ \ \text{and} \ \ \prod_{i=1}^{s} \lambda x_i = \prod_{i=1}^{s} \lambda x_{s+i}, \]
for every ${\lambda} \neq 0$. Moreover, note that
\[ \mathcal{Z}_{\lambda^{d+1} \cdot \vec{\varphi}_{\lambda^{-1}} } = \mathcal{Z}_{ \vec{\varphi}_{\lambda^{-1}}  } = \lambda \cdot \mathcal{Z}_{ \vec{\varphi}}. \]
Hence, upon dilation by an appropriate natural number $\lambda$, we may assume that $\vec{\varphi}$ is an element of $\mathbb{Z}[x]^{2s}$ as well as that $A_1, \dots, A_{2s}$ are finite subsets of $\mathbb{Z}\setminus (\{0\} \cup \mathcal{Z}_{\vec{\varphi}})$. 
\par

Next, let $X$ be a finite subset of $\mathbb{N}^2$. Then, there exists a sufficiently large distinct prime number $p$ such that the map $\varrho_{X} : X  \to \mathbb{N}$, defined as
\[ \varrho_{X}(  p_1^{\alpha_1} \dots p_r^{\alpha_r},  q_1^{\beta_1} \dots q_{t}^{\beta_t} ) =  p_1^{\alpha_1} \dots p_r^{\alpha_r} (  q_1^{\beta_1} \dots q_{t}^{\beta_t} )^{p} \]
for all primes $p_1, \dots, p_r,q_1, \dots, q_t$ and for all non-negative integers $\alpha_1, \dots, \alpha_r, \beta_1, \dots, \beta_t, r, t$, is bijective onto its image and satisfies the fact that for every $(x_1, y_1), \dots, (x_{2s}, y_{2s}) \in X$, we have
\[ x_1 \dots x_s = x_{s+1} \dots x_{2s} \ \text{and} \ y_1 \dots y_s = y_{s+1} \dots y_{2s} \]
if and only if
\[ \varrho_{X}(x_1, y_1) \dots \varrho_{X}(x_s ,y_s) = \varrho_{X}(x_{s+1}, y_{s+1}) \dots \varrho_{X}(x_{2s}, y_{2s}) .\]
Moreover, we can further define the logarithmic map $\psi_{X} : \varrho_{X}(X) \to [0, \infty)$ by writing $\psi_{X}(n) = \log n$ for every $n \in \varrho_X(X)$. Note that $\psi_{X}$ is bijective onto its image and satisfies the fact that for every $z_1, \dots, z_{2s} \in \varrho_X(X)$, we have that
\[ z_1 \dots z_{s} = z_{s+1} \dots z_{2s} \ \text{if and only if} \ \psi_{X}(z_1) + \dots + \psi_{X}(z_{s}) = \psi_{X}(z_{s+1}) + \dots + \psi_{X}(z_{2s}) . \]          
\par

With these preliminary manoeuvres finished, we will now proceed to prove our proposition.  First, we will show that $\eqref{eqz2}$ implies $\eqref{eqz1}$, and then we will prove $\eqref{eqz2}$. In order to prove the first part,  note that for each $1 \leq i \leq 2s$, we may partition 
\[ \mathbb{R}\setminus (\{0\}\cup \mathcal{Z}_{\varphi_i}) = I_{i,1}  \cup \dots \cup I_{i,{r_i}}, \] 
where $I_{i,1}, \dots, I_{i,r_i}$ are open, pairwise disjoint intervals such that $\varphi_i(x) \varphi_i(y)>0$ and $xy >0$ for all $x,y \in I_{i,j}$, where $1 \leq j \leq r_i$ and $r_i \leq d+2$. Writing $A_{i,j} = A_{i} \cap I_{i,j}$ for each $1 \leq i \leq 2s$ and $1 \leq j \leq r_i$, we see that
\[ J_{s, \frak{a}, \vec{\varphi} }(A_1, \dots, A_{2s}) = \sum_{1 \leq j_1\leq r_1} \dots \sum_{1 \leq j_{2s} \leq r_{2s}}  J_{s, \frak{a}, \vec{\varphi}}(A_{1,{j_1}}, \dots, A_{{2s},{j_{2s}}})
 , \]
 where we have crucially used the fact that $A_i = \cup_{1 \leq j \leq r_i} A_{i, j}$ for each $1 \leq i \leq 2s$, which itself follows from the hypothesis that $A_1, \dots, A_{2s}$ are subsets of $\mathbb{Q} \setminus (\{0\} \cup \mathcal{Z}_{\vec{\varphi}})$.  Combining the preceding expression with $\eqref{eqz2}$ and the fact that $r_1, \dots, r_{2s} \leq d+2$ then delivers $\eqref{eqz1}$.

We will now prove $\eqref{eqz2}$,  and so, we define the set $X_i = \{ ( |a| , |\varphi_i(a)| ) : a \in A_i \}$ for each $1 \leq i \leq 2s$. Note that $J_{s, \frak{a}, \vec{\varphi}}(A_1, \dots, A_{2s})$ is bounded by the number of solutions to the system
\[ |a_1| \dots |a_{s}| = |a_{s+1}| \dots |a_{2s}| \ \text{and} \ |\varphi_1(a_1)| \dots |\varphi_s(a_s) | = |\varphi(a_{s+1})| \dots |\varphi(a_{2s})| , \]
with $a_i \in A_i$ for each $1 \leq i \leq 2s$, where each solution is being counted with the weight $\frak{a}(a_1) \dots \frak{a}(a_{2s})$. Consequently, letting $X = \cup_{1 \leq i \leq 2s} X_i$ and $\sigma(x) = \psi_{X}(\varrho_{X}(x))$ for every $x \in X$, we see that
\[ J_{s, \frak{a}, \vec{\varphi}}(A_1, \dots, A_{2s}) \leq E_{s,\frak{a}}(\sigma(X_1,) \dots, \sigma(X_{2s})) \leq E_{s, \frak{a}}(\sigma(X_1))^{1/2s} \dots E_{s, \frak{a}}(\sigma(X_{2s}))^{1/2s} ,\]
where the last inequality follows from Lemma $\ref{reprove}$. By the definition of $\sigma$, we see that for each $1 \leq i \leq 2s$, the quantity $E_{s,\frak{a}}(\sigma(X_i))$ is equal to the number of solutions to the system 
\[ |a_1| \dots |a_{s}| = |a_{s+1}| \dots |a_{2s}| \ \text{and} \ |\varphi_i(a_1)| \dots |\varphi_i(a_s)| = |\varphi_i(a_{s+1})|\dots |\varphi_i(a_{2s})| , \]
with $a_1, \dots, a_{2s} \in A_i$, where each such solution is counted with the weight $\frak{a}(a_1) \dots \frak{a}(a_{2s})$. This, in turn, equals $J_{s, \frak{a}, \varphi_i}(A_i)$ since the functions $x$ and $\varphi_i(x)$ do not change signs as $x$ varies in $A_i$. Combining this with the preceding discussion, we get that
\[ J_{s, \frak{a}, \vec{\varphi}}(A_1, \dots, A_{2s}) \leq J_{s, \frak{a}, \varphi_1}(A_1)^{1/2s} \dots J_{s, \frak{a}, \varphi_{2s}}(A_{2s})^{1/2s} ,\]
which is the desired bound.
\end{proof}

Our final result in this section allows us to prove various other relations between the above type of mixed energies. 
 
\begin{lemma} \label{wm}
Let $d, s, r$ be natural numbers, let $\vec{\varphi}  \in (\mathbb{Q}[x])^{2s}$ satisfy $\deg \varphi_i \leq d$ for each $1 \leq i \leq 2s$. Then for all finite subsets $A_1, \dots, A_{r}$ of $\mathbb{Q}\setminus (\{0\} \cup \mathcal{Z}_{\vec{\varphi}})$ and for every function $\frak{a}: \mathbb{Q} \to [0,\infty)$, we have that
\begin{equation} \label{doj}
 J_{s, \frak{a}, \vec{\varphi} }(A_1 \cup \dots \cup A_{r}) \leq (d+2)^{2s}  r^{2s} \sup_{ 1\leq i \leq r} \sup_{1 \leq j \leq 2s} J_{s, \frak{a}, \varphi_j}(A_i).
 \end{equation}
For every finite subset $A$ of $\mathbb{Q}\setminus (\{0\}\cup\mathcal{Z}_{\varphi_1})$, for every $1 \leq l < s$ and for the function $\frak{a}(x) = 1$ for each $x \in \mathbb{R}$, we have
\begin{equation} \label{doj3}
 J_{s, \frak{a}, \varphi_1}(A) \ll_{s,d} |A|^{2s - 2l } J_{l, \frak{a}, \varphi_1}(A).
 \end{equation}
\end{lemma}

\begin{proof}

As before, $\eqref{doj}$ follows from Proposition $\ref{gvup}$ in a straightforward manner since
\begin{align*}
 J_{s, \frak{a}, \vec{\varphi} }(A_1 \cup \dots \cup A_r)  
 & = \sum_{1 \leq  i_1, \dots, i_{2s} \leq r}J_{s, \frak{a}, \vec{\varphi} }(A_{i_1}, \dots, A_{i_{2s}}) \\
& \leq (d+2)^{2s} \sum_{1 \leq  i_1, \dots, i_{2s} \leq r} \prod_{j=1}^{2s}   J_{s, \frak{a}, \varphi_j}(A_{i_j})^{1/2s} \\
 & = (d+2)^{2s} \prod_{j=1}^{2s} \Big(\sum_{i=1}^{r} J_{s, \frak{a}, \varphi_j}(A_i)^{1/2s}\Big) \\
& \leq (d+2)^{2s} r^{2s} \sup_{1 \leq i \leq r} \sup_{1 \leq j \leq 2s} J_{s, \frak{a}, \varphi_j}(A_i).
\end{align*}
\par

We will now outline the proof of $\eqref{doj3}$. As in the proof of Proposition $\ref{gvup}$, we see that upon losing a factor of $O_{s,d}(1)$, it suffices to consider the case when $A \subseteq I$, where $I$ is some interval such that the functions $\varphi_1(x)$ and $x$ do not change signs as $x$ varies in $I$. As before, this allows us to construct a finite set $Z \subseteq (0, \infty)$ such that $|Z| = |A|$ and $J_{t, \frak{a}, \varphi_1}(A) = E_{t, \frak{a}}(Z)$ for each $1 \leq t \leq s$, whereupon, a straightforward application of Lemma $\ref{reprove}$ delivers the estimate
\[ J_{s, \frak{a}, \varphi_1}(A) = E_{s, \frak{a}}(Z) \leq |Z|^{2s - 2l} E_{l, \frak{a}}(Z) = |A|^{2s - 2l} J_{l, \frak{a}, \varphi_1}(A), \]
consequently finishing our proof of Lemma $\ref{wm}$. 
\end{proof}

It is worth noting that in the hypotheses of Proposition $\ref{gvup}$ and Lemma $\ref{wm}$, there are no lower bounds for $\deg \varphi_i$, whenceforth, we may even choose the polynomials $\varphi_{1}, \dots, \varphi_{2s}$ to be constant functions. In particular, if $\varphi_1(x) = \dots = \varphi_{2s}(x) = 1$ for every $x \in \mathbb{R}$, then we see that $J_{s, \frak{a}, \vec{\varphi}}(A) = M_{s, \frak{a}}(A)$ for every finite set $A \subseteq \mathbb{Q}$. This implies that Proposition $\ref{gvup}$ and Lemma $\ref{wm}$ hold true when the weighted mixed energies $J_{s, \frak{a}, \vec{\varphi}}(A)$ are replaced by weighted multiplicative energies $M_{s, \frak{a}}(A)$. 



\section{Chang's lemma for additive equations and query complexity}

Let $\varphi(x)$ be a polynomial in $\mathbb{Z}[x]$ of degree $d$, for some $d \in \mathbb{N}$. Given finite sets $A \subseteq \mathbb{N}$ and $V \subseteq \mathbb{Z}$ and some prime $p$, we write 
\[ A_{p, V} =  \{ a \in A \  | \ \nu_{p}(a) \in V \},  \]
where for any $n \in \mathbb{N}$, the quantity $\nu_{p}(n)$ denotes the largest exponent $m \in \mathbb{Z}$ such that $p^m$ divides $n$.  Moreover, we set $\nu_p(0) = \infty$. When $V = \{l\}$ for some $l \in \mathbb{Z}$, we write $A_{p,l} = A_{p,V}$, and we define $\nu_p(A) = \{ \nu_{p}(a) \ | \ a \in A\}.$ Furthermore, we will frequently use the following straightforward application of H\"{o}lder's inequality, that is, given natural numbers $r,s$ and bounded functions $f_1, \dots, f_r : [0,1] \to \mathbb{C}$, we have
\begin{equation} \label{hl3}
 \int_{[0,1)}  |\sum_{i=1}^{r} f_i(\alpha) |^{2s} d \alpha
 \leq r^{2s-1} \sum_{i=1}^{r} \int_{[0,1)}  |f_{i}(\alpha)|^{2s} d\alpha
 \leq r^{2s} \max_{1 \leq i \leq r} \int_{[0,1)} |f_i(\alpha)|^{2s} d\alpha .
\end{equation}
\par

Our main object of study in this section would be $E_{s, \frak{a}, \varphi}(A)$. We note that these weighted additive energies can be represented as moments of various types of exponential sums, and so, we define the function
\[ f_{\frak{a}, \varphi}(A; \alpha) =  \sum_{a \in A} \frak{a}(a) e(\alpha \varphi(a) ) \]
when $A$ is a non-empty set and we set $f_{\frak{a}, \varphi}(A; \alpha) = 0$ if $A$ is an empty set. By orthogonality, we see that 
\begin{equation} \label{orth}
 \int_{[0,1)} | f_{\frak{a}, \varphi}(A; \alpha) |^{2s}  d\alpha  
 = E_{s, \frak{a}, \varphi}(A) .
 \end{equation}
Moreover, note that 
\[ f_{\frak{a}, \varphi}(A; \alpha) =  \sum_{l \in \nu_p(A)} f_{\frak{a}, \varphi}(A_{p,l}; \alpha) ,\]
which then combines with $\eqref{hl3}$ to give us
\[   \int_{[0,1)} | f_{\frak{a}, \varphi}(A; \alpha) |^{2s}  d\alpha \leq |\nu_p(A)|^{2s-1} \sum_{l\in \nu_p(A)}  \int_{[0,1)} | f_{\frak{a}, \varphi}(A_{p, l}; \alpha)  |^{2s} d\alpha .
\]
The following lemma essentially allows us to upgrade the factor $|\nu_{p}(A)|^{2s-1}$ in the above inequality to a factor of  $O_{s,d}(|\nu_{p}(A)|^{s-1})$.

\begin{lemma} \label{chang}
Let $p$ be a prime number, let $A \subseteq \mathbb{N}$ be a finite set and let $\frak{a}: \mathbb{N} \to [0, \infty)$ be a function supported on $A$. Then 
\[  \Big( \int_{[0,1)} | f_{\frak{a}, \varphi}(A; \alpha) |^{2s}  d\alpha \Big)^{1/s}
 \leq (d^2+2)^{4} (2s)^2
 \sum_{n \in \nu_p(A)} \Big(  \int_{[0,1)} | f_{\frak{a}, \varphi}(A_{p, n}; \alpha)  |^{2s} d\alpha \Big)^{1/s} 
\]
\end{lemma}

\begin{proof}
For ease of notation, we will write $f(B; \alpha) = f_{\frak{a}, \varphi}(B; \alpha)$ for every $B \subseteq \mathbb{Z}$ and $\alpha \in \mathbb{R}$, thus suppressing the dependence on $\frak{a}$ and $\varphi$. 
We begin by restricting our analysis to the case when $\varphi(0) = 0$, since the additive equation $x_1 + \dots + x_s = x_{s+1} + \dots + x_{2s}$ is translation invariant. Thus, let $\varphi(x) = \sum_{i \in I} \beta_i x^i$ for some non-empty set $I \subseteq \{1,\dots, d\}$ and for some sequence $\{\beta_i\}_{i \in I}$ of non-zero integers. Moreover,  we define 
\[ X = \bigg\{ \frac{\nu_{p}(\beta_j) - \nu_p(\beta_i)}{i-j} \ : \ i , j \in I \ \text{and} \ i \neq j \bigg\}, \]
whereupon, we have $|X| \leq |I|^2 \leq d^2$. Denoting $r = |X|$, we write the elements of $X$ in increasing order as $x_1 < \dots < x_r$, and we decompose the set $\nu_p(A)$ as $\nu_p(A) = U_0 \cup \dots U_{r+1}$, where $U_i = (x_i, x_{i+1}) \cap \nu_p(A)$ for each $1 \leq i \leq r-1$ and $U_0 = (-\infty, x_1) \cap \nu_p(A)$ and $U_{r} = (x_{r}, \infty) \cap \nu_p(A)$ and $U_{r+1} = X \cap \nu_p(A)$. We note that the sets $U_0, \dots, U_{r+1}$ are pairwise disjoint and that $|U_{r+1}| \leq d^2$. 
\par

We begin by applying $\eqref{hl3}$ in order to deduce that
\begin{equation} \label{soc}
 \int_{[0,1)} | f (A; \alpha) |^{2s}  d\alpha \leq (r+2)^{2s} \max_{0 \leq i \leq r+1} \bigg\{  \int_{[0,1)} | f (A_{p, U_i}  ; \alpha) |^{2s} \bigg\} .
 \end{equation}
Note that if the set $U_{r+1}$ maximises the right hand side above, then we may apply $\eqref{hl3}$ again to deduce the desired claim. Thus, we may assume that $U_i$ maximises the right hand side above for some $0 \leq i \leq r$. We now claim that for any such $i$ there exists $j = j(i) \in I$ such that for every $a \in A_{p, U_i}$, we have $\nu_{p}( \varphi(a)) = \nu_p( \beta_j a^j)$. This arises from combining the fact that 
\[ \nu_p ( (p^m)^{i} \beta_i ) < \nu_p ( (p^m)^j \beta_j) \ \ \text{if and only if} \ \ m< ( \nu_{p}(\beta_j) - \nu_p(\beta_i) ) (i-j)^{-1},   \]
and that $m$ lies in the set $U_i$, which itself is contained in precisely one of the intervals $(-\infty, x_1), (x_1, x_2), \dots, (x_r, \infty)$. 
\par

Now suppose that $a_1, \dots, a_{2s} \in A_{p, U_i}$ satisfy 
\begin{equation} \label{nin}
\varphi(a_1) + \dots + \varphi(a_{s}) = \varphi(a_{s+1}) + \dots + \varphi(a_{2s}). 
\end{equation}
Our next claim is that there exist distinct $k_1, k_2 \in \{1, \dots, 2s\}$ such that $a_{k_1}, a_{k_2} \in A_{p,n}$ for some $n \in U_i$. If this was not so, then writing $k_0$ to be the distinct $k$ which minimises $\nu_{p}(a_k)$, we have that
\[ \nu_p(0) = \nu_p(\varphi(a_1) + \dots - \varphi(a_{2s})) = \nu_{p}( \beta_j(a_1^j + \dots - a_{2s}^j)) = \nu_{p}(\beta_j a_{k_0}^j) < \infty, \]
which delivers a contradiction.
Combining this claim with the orthogonality relation $\eqref{orth}$ and applying triangle inequality, we get that
\[   \int_{[0,1)} | f(A_{p,U_i} ; \alpha) |^{2s}  d\alpha  \leq (2s)^{2}   \sum_{n \in U_i} \int_{[0,1)} | f(A_{p,U_i} ; \alpha) |^{2s-2} | f(A_{p,n} ; \alpha) |^2 d\alpha .  \]
Employing H\"{o}lder's inequality now enables us to bound the right hand side by 
\[  (2s)^{2}   \sum_{n \in U_i}
\Big( \int_{[0,1)} | f(A_{p,U_i} ; \alpha) |^{2s} d\alpha \Big)^{1-1/s}
\Big(\int_{[0,1)} |f(A_{p,n} ; \alpha) |^{2s} d\alpha \Big)^{1/s},
\]
and subsequently, we have that
\[ \Big(\int_{[0,1)} | f(A_{p,U_i} ; \alpha) d\alpha  ) |^{2s}  \Big)^{1/s} \leq (2s)^{2}  \sum_{n \in U_i} \Big(\int_{[0,1)} |f(A_{p,n} ; \alpha) |^{2s} d\alpha \Big)^{1/s}.  \]
Substituting this into $\eqref{soc}$ finishes our proof.
\end{proof}


 In order to present the next lemma, we record some further notation, and thus,
%
given a finite set $A \subseteq \mathbb{Z}$, we define the \emph{query complexity} $q(A)$ to be the minimal $t \in \mathbb{N}$ such that there are functions $f_1, \dots, f_{t-1} : \mathbb{Z} \to \mathbb{P}$ and a fixed prime number $p_1$ satisfying the fact that the vectors $\{(\nu_{p_1}(a), \dots, \nu_{p_t}(a))\}_{a \in A}$ are pairwise distinct, where the prime numbers $p_2, \dots, p_t$ are defined recursively by setting $p_i = f_{i-1}(\nu_{p_{i-1}}(a))$, for each $2 \leq i \leq t$.

\begin{lemma} \label{ayay}
Let $s,t$ be natural numbers, let $A \subseteq \mathbb{N}$ be a finite set such that $q(A) = t$ and  let $\frak{a}: \mathbb{N} \to [0, \infty)$ be a function supported on $A$. Then we have that
\[  E_{s, \frak{a}, \varphi}(A) ^{1/s}
 \leq (d^2+2)^{4t} (2s)^{2t} \sum_{a \in A} \frak{a}(a)^2. 
\]
\end{lemma}

\begin{proof}
For the sake of exposition, we will write $E(X) = E_{s, \frak{a}, \varphi}(X)$ for any subset $X$ of $A$.
Our aim is to prove our lemma by induction on $t$, and so, we first consider the case when $t=1$. In this case, there must exist some prime $p$ such that $\nu_{p}(a)$ is distinct for each $a \in A$. Applying Lemma $\ref{chang}$ along with the orthogonality relation $\eqref{orth}$, we infer that
\[ E(A)^{1/s} \leq (d^2 + 2)^4 (2s)^2 \sum_{n \in \nu_p(A)} E(A_{p,n})^{1/s}. \]
Moreover, since for each $n \in \nu_p(A)$, the set $A_{p,n}$ may be written as $A_{p,n} = \{ a_n\}$ for some unique $a_n \in A$, we see that $E(A_{p,n}) = \frak{a}(a_n)^{2s}$. Substituting this in the above expression gives the desired bound when $t=1$.
\par

We now assume that $t >1$, whence, there exist functions $f_1, \dots, f_{t-1} : \mathbb{Z} \to \mathbb{P}$ and a fixed prime number $p_1$ satisfying the fact that the vectors $\{(\nu_{p_1}(a), \dots, \nu_{p_t}(a))\}_{a \in A}$ are pairwise distinct, where the prime numbers $p_2, \dots, p_t$ are defined recursively by setting $p_i = f_{i-1}(\nu_{p_{i-1}}(a))$, for each $2 \leq i \leq t$.  As before, upon applying an amalgamation of Lemma $\ref{chang}$ and $\eqref{orth}$, we see that
\[ E(A)^{1/s} \leq (d^2 + 2)^4 (2s)^2 \sum_{n \in \nu_{p_1}(A)} E(A_{p_1, n})^{1/s} . \]
Note that for each $n \in \nu_{p_1}(A)$, the set $A_{p_1, n}$ satisfies $q(A_{p_1, n}) \leq t-1$, and so, we may apply the inductive hypothesis for each set $A_{p_1,n}$ to get the inequality
\[ E(A)^{1/s} \leq (d^2 + 2)^{4t} (2s)^{2t} \sum_{n \in \nu_{p_1}(A)} \sum_{a \in A_{p_1, n} } \frak{a}(a)^2 \leq   (d^2 + 2)^{4t} (2s)^{2t}   \sum_{ a \in A} \frak{a}(a)^2.  \]
This finishes the inductive step, and subsequently, the proof of Lemma $\ref{ayay}$ as well.
\end{proof}


\section{Chang's lemma for multiplicative equations and an averaging argument}

Our main aim of this section is to prove a multiplicative analogue of Lemma $\ref{ayay}$, and this forms the content of the following result.

\begin{lemma} \label{tst}
Let $s, d$ be natural numbers, let $\varphi \in \mathbb{Z}[x]$ be a polynomial of degree $d$, let $A$ be a finite subset of $\mathbb{Z} \setminus ( \{0\} \cup \mathcal{Z}_{\varphi})$ and let $\frak{a}: \mathbb{R} \to [0, \infty)$ be a function. Then 
\[ M_{s, \frak{a}, \varphi}(A) \ll_{s, d}|A^{(s)}/A^{(s)}| (d+3)^{16 q(A) s} (2s)^{2 q(A)s} \Big( \sum_{a \in A} \frak{a}(a)^2\Big)^s . \]
\end{lemma}

We begin this endeavour  by proving an analogue of Lemma $\ref{ayay}$ for $J_{s, \frak{a}, \varphi}(A)$.

\begin{lemma} \label{ob2}
Let $p$ be a prime number, let $\varphi$ be a polynomial of degree $d$ such that $\varphi(0) \neq 0$,  let $\mathcal{I} \subseteq \mathbb{R}$ be some open interval such that $xy >0$ and $\varphi(x) \varphi(y) >0$ for every $x,y \in \mathcal{I}$, let $A \subseteq \mathcal{I}$ be a finite set and let $\frak{a}: \mathbb{N} \to [0, \infty)$ be a function. Then we have
\[  J_{s, \frak{a}, \varphi}(A)^{1/s}
 \leq (d+3)^{16} (2s)^2
 \sum_{n \in \nu_p(A)} J_{s, \frak{a}, \varphi}(A_{p,n})^{1/s} 
\]
\end{lemma}

\begin{proof}
For ease of exposition, we write $J(A) = J_{s, \frak{a}, \varphi}(A)$, thus suppressing the dependence on $s, \frak{a}, \varphi$. Moreover, we import various notation from the proof of Lemma $\ref{chang}$, and so, suppose that $\varphi(x) = \sum_{i \in I} \beta_i x^{i}$ for some $\{0,d\} \subseteq I \subseteq \{0,1,\dots, d\}$ and for some sequence $\{ \beta_i\}_{i \in I}$ of non-zero integers. We define the elements $x_1 < \dots < x_r$ of the set $X$, and the sets $U_0, \dots, U_{r+1}$ precisely as in the proof of Lemma $\ref{chang}$. The only difference is that in this case, we will have that $|U_{r+1}| \leq r = |X| \leq |I|^2 = (d+1)^2$.
\par

Our first claim is that for any fixed $0 \leq i \leq r$, whenever $a_1, \dots, a_{2s} \in A_{p, U_i}$ satisfy
\[ a_1 \dots a_s = a_{s+1} \dots a_{2s} \ \ \text{and} \ \ \varphi(a_1) \dots \varphi(a_s) = \varphi(a_{s+1}) \dots \varphi(a_{2s}), \]
then there exist distinct $k_1, k_2 \in \{1, \dots, 2s\}$ such that $\nu_p(a_{k_1}) = \nu_p(a_{k_2})$. We prove this by contradiction, and so, suppose $a_1, \dots, a_{2s} \in A_{p, U_i}$ satisfy the above system of equations and let $k$ and $k'$ be the unique elements in $\{1, \dots, 2s\}$ for which $\nu_p(a_k)$ is minimal and $\nu_p(a_{k'})$ is maximal. Moreover, since $a_1, \dots, a_{2s} \in A_{p,U_i}$, there exist distinct $j_1, j_2 \in I$ such that 
\[ \nu_{p}(\beta_{j_1} a_l^{j_1}) < \nu_p(\beta_{j_2} a_l^{j_2} ) < \nu_p(\beta_j a_l^{j}) \]
for every $j \in I \setminus \{j_1, j_2\}$ and for every $1 \leq l \leq 2s$. This implies that
\[ \nu_p(0) = \nu_p( \varphi(a_1) \dots \varphi(a_s) - \varphi(a_{s+1}) \dots \varphi(a_{2s})  ) =  \nu_p( T_1 - T_2 ),              \]
where 
\[ T_1 = (\beta_{j_1}^s\prod_{l=1}^{s}  a_l^{j_1} )(1 + \beta_{j_2} \beta_{j_1}^{-1} \sum_{l=1}^{s} a_l^{j_2- j_1}  )  \ \text{and} \ T_2 = (\beta_{j_1}^s\prod_{l=s+1}^{2s}  a_l^{j_1} )(1 + \beta_{j_2} \beta_{j_1}^{-1} \sum_{l=s+1}^{2s} a_l^{j_2- j_1}  ). \]
In the case when $j_2 > j_1$, utilising the fact that $a_1 \dots a_s = a_{s+1} \dots a_{2s}$, we get
\[ \nu_p(T_1 - T_2) = \nu_p( \beta_{j_2} \beta_{j_1}^{s-1} (a_1\dots a_s)^{j_1}  \sum_{l=1}^{s} (a_l^{j_2 - j_1} - a_{l+s}^{j_2 - j_1}) )  = \nu_p(  \beta_{j_2} \beta_{j_1}^{s-1} (a_1\dots a_s)^{j_1} a_{k}^{j_2 - j_1}) < \infty, \]
which implies that $\nu_p(0) < \infty$, thus delivering the desired contradiction. Similarly if $j_1 > j_2$, then we get that
\[ \nu_p(T_1 - T_2) = \nu_p( \beta_{j_2} \beta_{j_1}^{s-1} (a_1\dots a_s)^{j_1}  \sum_{l=1}^{s} (a_l^{j_2 - j_1} - a_{l+s}^{j_2 - j_1}) )  = \nu_p(  \beta_{j_2} \beta_{j_1}^{s-1} (a_1\dots a_s)^{j_1} a_{k'}^{j_2 - j_1}) < \infty, \]
which subsequently implies that $\nu_p(0) < \infty$ as well,  providing the required contradiction.
\par

With this claim in hand, we now use $\eqref{doj}$ to deduce that 
\[ J_{s, \frak{a}, \varphi}(A) \leq (d+2)^{2s} (r+2)^{2s} \max_{0 \leq i \leq  r+1} J_{s, \frak{a}, \varphi}(A_{p,U_i}).  \]
If the maximum is attained by the term corresponding to the set $A_{p, U_{r+1}}$, we may then combine a second application of $\eqref{doj}$ along with the facts that $A_{p, U_{r+1}} = \cup_{n \in U_{r+1}}A_{p,n}$ and $|U_{r+1}| \leq r \leq (d+1)^2$ to deliver the desired bound. Thus, it suffices to consider the case when the maximum in the preceding expression is attained for some fixed $i \in \{0,1,\dots, r\}$. In this case, we may use our aforementioned claim to note that
\[ J_{s, \frak{a}, \varphi}(A_{p,U_i}) \leq  \sum_{n \in U_i} (2s)^{2} 
\max J_{s, \frak{a}, \varphi}(B_1, B_2, \dots, B_{2s}),\]
where the maximum is taken over all choices of sets $B_1, \dots, B_{2s}$ such that precisely two of these sets are $A_{p,n}$ and the rest are $A_{p, U_i}$. We may now apply $\eqref{eqz2}$ to deduce that
\[   J_{s, \frak{a}, \varphi}(A_{p,U_i}) \leq  (2s)^2 \sum_{n \in U_i}    J_{s, \frak{a}, \varphi}(A_{p,U_i})^{1- 1/s} J_{s, \frak{a}, \varphi}(A_{p,n})^{1/s} . \]
Combining this with the preceding discussion delivers the bound
\[ J_{s, \frak{a}, \varphi}(A)^{1/s}  \leq (d+2)^2(r+2)^{2} (2s)^2 \sum_{n \in \nu_p(A)} J_{s, \frak{a}, \varphi}(A_{p,n})^{1/s} ,\]
which, in turn, combines with the fact that $r \leq (d+1)^2$ give us the desired result.
\end{proof}

%

As in \S4, the above lemma may be iterated to furnish the following result. 

\begin{lemma} \label{ob2}
Let $\varphi$ be a polynomial of degree $d$ such that $\varphi(0) \neq 0$, let $\mathcal{I} \subseteq \mathbb{R}$ be some open interval such that $xy >0$ and $\varphi(x) \varphi(y) >0$ for every $x,y \in \mathcal{I}$, let $A \subseteq \mathcal{I}$ be a finite set such that $q(A) = t$ and let $\frak{a}: \mathbb{N} \to [0, \infty)$ be a function. Then we have
\[  J_{s, \frak{a}, \varphi}(A)^{1/s}
 \leq (d+3)^{16t} (2s)^{2t}
\sum_{a \in A} \frak{a}(a)^2 
\]
\end{lemma}

We will now combine this with an averaging argument to deliver bounds for $M_{s, \frak{a}, \varphi}(A)$, and in this endeavour, we first prove a more general lemma that allows us to bound the number of solutions to a system of equations by inserting further auxiliary equations. Thus, given natural numbers $d_1, d_2$ and functions $f : \mathbb{R} \to \mathbb{R}^{d_1}$ and $g : \mathbb{R} \to \mathbb{R}^{d_2}$ and $\frak{a} : \mathbb{R} \to [0, \infty)$, we let
\[ E_{s, \frak{a}}(A;f,g) = \sum_{a_1, \dots, a_{2s} \in A} \frak{a}(a_1) \dots \frak{a}(a_{2s}) \mathds{1}_{f(a_1) + \dots - f(a_{2s}) = 0} \mathds{1}_{g(a_1) + \dots - g(a_{2s}) = 0}. \]
In particular, this counts the number of solutions to the system 
\[ \sum_{i=1}^{s}( f(x_i) - f(x_{i+s}) )= \sum_{i=1}^{s} ( g(x_{i}) - g(x_{i+s}) ) = 0\]
where each solution $(x_1, \dots, x_{2s}) \in A^{2s}$ is being counted with weights $\frak{a}(x_1) \dots \frak{a}(x_{2s})$. 
Similarly, we define 
\[ E_{s, \frak{a}}(A;g) = \sum_{a_1, \dots, a_{2s} \in A} \frak{a}(a_1) \dots \frak{a}(a_{2s})  \mathds{1}_{g(a_1) + \dots - g(a_{2s}) = 0}. \]

\begin{lemma} \label{trut} 
Let $s, d_1, d_2$ be natural numbers and let $f : \mathbb{R} \to \mathbb{R}^{d_1}$ and $g : \mathbb{R} \to \mathbb{R}^{d_2}$ and $\frak{a} : \mathbb{R} \to [0, \infty)$ be functions. Then for any finite set $A$ of real numbers, we have that 
\[  E_{s, \frak{a}}(A; g) \leq |s f(A) - sf(A)| E_{s, \frak{a}}(A;f,g) . \]
\end{lemma}

We will now show that Lemmata $\ref{trut}$ and $\ref{ob2}$ imply Lemma $\ref{tst}$.  

\begin{proof}[Proof of Lemma $\ref{tst}$]
We first partition the set $A$ into finite sets $A_1, \dots, A_r$, for some $r \ll_{d} 1$, such that for each $1 \leq j \leq r$, the set $A_j$ lies in some interval $I_j$ for which we have $x y >0$ and $\varphi(x) \varphi(y) >0$ for every $x,y \in I_j$. Noting the remark following Lemma $\ref{wm}$, we may apply Lemma $\ref{wm}$ for the multiplicative energies $M_{s, \frak{a}, \varphi}(A_1 \cup \dots \cup A_r)$, whereupon, we see that it suffices to bound $M_{s, \frak{a}, \varphi}(A_i)$ for each $1 \leq i \leq r$ individually. Moreover, for each such $A_i$, we can let $f(x) = \log x$ and $g(x) = \log \varphi (x)$ and apply Lemma $\ref{trut}$ to deduce that
\[ M_{s, \frak{a}, \varphi}(A) \ll_{s,d} |A^{(s)}/A^{(s)}| J_{s, \frak{a}, \varphi}(A).  \]
We obtain the required bound by substituting the estimate presented in the conclusion of Lemma $\ref{ob2}$.
\end{proof}

We end this section by recording the proof of Lemma $\ref{trut}$.

\begin{proof}[Proof of Lemma $\ref{trut}$]

 Let $Y = \{ (f(a), g(a)) \ | \ a \in A\}$ and let $G$ be the additive abelian group generated by elements of $Y$. Note that
\begin{align} \label{ugh}
 E_{s}(A;g) & = \sum_{a_1, \dots , a_{2s} \in A} \frak{a}(a_1) \dots \frak{a}(a_{2s}) \mathds{1}_{g(a_1) + \dots - g(a_{2s})=0}  \nonumber \\
 & = \sum_{\vec{n} \in sf(A) - sf(A)}  \sum_{a_1, \dots , a_{2s} \in A} \frak{a}(a_1) \dots \frak{a}(a_{2s}) \mathds{1}_{g(a_1) + \dots - g(a_{2s}) = 0} \mathds{1}_{f(a_1) + \dots - f(a_{2s}) = \vec{n}} . 
 \end{align}
We now introduce some further notation, and so, given functions $h_1, h_2 : G \to \mathbb{R}$, we define the function $h_1 * h_2 : G \to \mathbb{R}$ by setting
\[ (h_1*h_2)(\vec{x}) = \sum_{\vec{y} \in G} h_1(\vec{x} - \vec{y}) h_2(\vec{y}) , \]
for every $\vec{x} \in G$. Next, for every $s \geq 2$, we define the function $h_1*_{s} h_1= h_1 * (h_1 *_{s-1} h_1)$ as well as $h_1*_{1} h_1 = h_1* h_1$. Moreover, for any finite set $X$, we define the function $\frak{b}: G \to  [0, \infty)$ by letting $\frak{b}(x,y) = \frak{a}(a)$ whenever $(x,y) = (f(a), g(a))$ for some $a \in A$, and we set $\frak{b}(x,y) = 0$ for every other choice of $(x,y) \in G$. Finally, we let $\frak{b}': G \to  [0, \infty)$ be a function defined as $\frak{b}'(x,y) = \frak{b}(-x,-y)$ for every $(x,y) \in G$.
\par

The above notation allows us to rewrite $\eqref{ugh}$ as 
\[ E_{s, \frak{a}}(A;g) = \sum_{\vec{n} \in sf(A) - sf(A)} (\frak{b} *_{s-1} \frak{b} *\frak{b}' *_{s-1} \frak{b}' )(\vec{n}, 0) .\]
Applying Young's convolution inequality, we see that
\[ \sup_{(x,y) \in G} |(\frak{b} *_{s-1} \frak{b} *\frak{b}' *_{s-1} \frak{b}')(x,y)| \leq \Big(\sum_{(x,y) \in G} (\frak{b} *_{s-1} \frak{b})(x,y)^2 \Big)^{1/2}  \Big( \sum_{(x,y) \in G} (\frak{b}' *_{s-1} \frak{b}')(x,y)^2 \Big)^{1/2} . \]
Double counting then implies that 
\[  \sum_{(x,y) \in G} (\frak{b} *_{s-1} \frak{b})(x,y)^2 = \sum_{(x,y) \in G} (\frak{b}' *_{s-1} \frak{b}')(x,y)^2 = E_{s, \frak{a}}(A;f,g), \]
which subsequently combines with the preceding discussion to give us the bound
\[ E_{s, \frak{a}}(A;g) \leq |sf(A) - sf(A)| E_{s, \frak{a}}(A;f,g), \]
thus concluding our proof of Lemma $\ref{trut}$.
\end{proof}


\section{Inverse results from additive combinatorics}

We utilise this section to present the various inverse results from additive combinatorics that we shall employ in our proof of Theorem $\ref{th3}$, and we begin this endeavour by recording some definitions. For any subsets $A, B$ of some abelian group $G$, define
\[ E(A,B) = |\{ (a_1, a_2, b_1, b_2) \in A^2 \times B^2 \ | \ a_1 + b_1 = a_2 + b_2\} \ \ \text{and} \ \ E(A) = E(A,A). \]
Moreover, throughout this part of the paper, we will denote $\mathcal{C}$ to be some large, absolute, computable, positive constant, which may change from line to line. 
\par 

First, we will record a classical result in additive combinatorics known as the Pl{\"u}nnecke--Ruzsa inequality \cite[Corollary 6.29]{TV2006},  see also \cite{Pe2012} for a shorter and simpler proof.

\begin{lemma} \label{pr21}
Let $A$ be a finite subset of some additive abelian group $G$. If $|A+A| \leq K|A|$, then for all non-negative integers $m,n$, we have
\[  |mA - nA| \leq K^{m+n}|A|. \]
\end{lemma}

We now present one of the main inverse results that we will use to prove Theorem $\ref{th3}$.

\begin{theorem} \label{th46}
Let $A \subseteq \mathbb{R}$ be a finite set, let $s\geq 16$ be an integer and let $E_{s}(A)  = |A|^{2s-1}/K$, for some $K \geq 1$. Then there exists $2 \leq s' \leq s$ and a finite, non-empty set $U' \subseteq s'A$ satisfying
\[ |U'| \ll K |A|  \ \ \text{and} \ \ \max_{x \in U'-A} |(A+x) \cap U'| \gg \frac{|A|}{K^{\mathcal{C}/\log s} } \]
such that for every $m,n \in \mathbb{N} \cup \{0\}$, we have
\[  |m U' - nU'| \ll_{m,n} K^{\mathcal{C}(m+n)/ \log s} |U'| . \]
\end{theorem}

We will prove this by utilising some of the ideas from \cite{Mu2021c},  and in particular, we will use the following result that may be deduced from the proof of Proposition 2.3 in \S8 of \cite{Mu2021c}.

\begin{Proposition} \label{prop1}
Let $\nu, \delta$ be positive real numbers such that $\nu \geq 1$ and let $s \geq 4$ be some even number. Moreover, suppose that $A \subseteq \mathbb{R}$ is a finite, non-empty set such that $E_{s}(A) \geq |A|^{2s - \nu}$. Then we either have $E_{s/2}(A) > |A|^{s - \nu + \delta},$ or there exists some non-empty subset $U' \subseteq (s/2)A$ satisfying
\[ |U'| \ll |A|^{\nu} \ \ \text{and} \ \ \max_{x \in U' -A} |(A+x) \cap U'| \gg  |A|^{1 - 82 \delta}   \]
such that for every $m, n \in \mathbb{N} \cup \{0\}$, we have
\[ |mU' - nU'| \ll_{m,n} |A|^{240(m+n) \delta} |U'|. \]
\end{Proposition}

We are now ready to prove Theorem $\ref{th46}$.

\begin{proof}[Proof of Theorem $\ref{th46}$]
%
%
Letting $l \geq 4$ be the integer satisfying $2^l \leq s < 2^{l+1}$,  we may use Lemma \ref{reprove} to deduce that
\[ E_{2^l}(A) \geq E_{s}(A) /|A|^{2s - 2^{l+1}} \geq |A|^{2^{l+1} - \nu} ,\]
where we write $|A|^{\nu-1} = K$. In particular, this means that
\[ |A|^{1-\nu } \leq \frac{E_{2^l}(A)}{E_{2}(A) |A|^{2^{l+1} -4} }= \prod_{i=2}^{l} \frac{E_{2^{i}}(A)}{E_{2^{i-1}}(A) |A|^{2^{i}} }, \]
whereupon, there exists some $2 \leq i \leq l$ such that
\[ E_{2^{i}} (A) \geq |A|^{2^i} E_{2^{i-1}}(A)  |A|^{(1 - \nu)/(l-1) } . \]
Defining $\nu'$ to be the real number such that $E_{2^i}(A) = |A|^{2^{i+1} - \nu'}$, we see that $\nu \geq \nu' \geq 1$ since
\[ |A|^{2^{i+1} - 1} \geq E_{2^i}(A) \geq E_{2^l}(A) |A|^{-2^{l+1} + 2^{i+1}} \geq |A|^{2^{i+1} - \nu} .\]
Furthermore, by the preceding discussion, we have that
\[ E_{2^{i-1}}(A) \leq |A|^{2^i - \nu' + (\nu - 1)/(l-1) } . \]
We may now apply Proposition $\ref{prop1}$ with $s = 2^i$ and $\delta = (\nu-1)/(l-1)$ to deduce the existence of some finite, non-empty set $U' \subseteq 2^{i-1}A$ satisfying
\[ |U'| \ll |A|^{\nu'} \ll |A|^{\nu} \ \ \text{and} \ \ \max_{x \in U'-A} |(A+x) \cap U'| \gg |A|^{1 - 82 (\nu-1)/(l-1)}  \gg |A|^{1 - 164 (\nu-1)/l} \]
such that for every $m , n \in \mathbb{N} \cup \{0\}$, we have
\[ |mU'-nU'| \ll_{m}  |A|^{240(m+n)(\nu - 1)/(l-1) }  |U'| \ll_{m} |A|^{480(m+n) (\nu-1)/l}|U'|. \]
We obtain the desired conclusion by noting that $(\log s)/2 \leq   l \leq \log s$ and substituting $|A|^{\nu - 1} = K$.
\end{proof}

Theorem \ref{th46} can be seen as an efficient many-fold version of a Balog--Szemer\'{e}di--Gowers type theorem,  see \cite{Sch2015} for more details about the latter.  We further note that a multiplicative analogue of Theorem \ref{th46} for finite sets $A \subseteq \mathbb{N}$ follows in a straightforward manner from Theorem \ref{th46} by considering the logarithmic map from $\mathbb{N}$ to $[0, \infty)$.

Our next goal in this section is to prove the following lemma which arises from adapting some of the methods from \cite{PZ2020}, see also \cite{Gr2022}.

\begin{lemma} \label{cmb}
Let $A, X$ be finite, non-empty subsets of $\mathbb{N}$ such that $|A \cdot X\cdot X| \leq K|X|$, for some $K \geq 1$. Then there exists a subset $B \subseteq A$ such that
\[ |B| \geq |A|/K \ \ \text{and}  \ \ q(B) \leq \log (2K) . \]
\end{lemma}

In our proof of the above result,  we  will closely follow ideas and definitions as recorded in \cite[Chapter $8$]{Gr2022},  and so, given $r \in \mathbb{N}$, we define the map $\pi_{i,r} : \mathbb{Z}^r \to \mathbb{Z}$ as $\pi_{i,r}(x_1, \dots, x_r) = x_i$ for every $(x_1, \dots , x_r) \in \mathbb{Z}^r$. We denote a finite set $X \subseteq \mathbb{Z}$ to be a \emph{quasicube} if $|X|=2$. Moreover, when $r \geq 2$ and $X \subseteq \mathbb{Z}^r$ is some finite set, then we write $X$ to be a \emph{quasicube} if $\pi_{r,r}(X) = \{y_1, y_2\}$ for some distinct $y_1, y_2 \in \mathbb{Z}$ as well as if the sets 
\[ \{ (x_1, \dots, x_{r-1}) \in \mathbb{Z}^{r-1} \ | \ (x_1, \dots, x_{r-1}, y_1) \in X\} \] 
and
 \[ \{ (x_1, \dots, x_{r-1}) \in \mathbb{Z}^{r-1} \ | \ (x_1, \dots, x_{r-1}, y_2) \in X\} \]
are also quasicubes. Morever, a subset of a quasicube is called a \emph{binary set}. 
\par

We define a set $V \subseteq \mathbb{Z}^r$ to be an \emph{axis aligned subspace} if $V = X_1 \times \dots \times X_r$, where for every $1 \leq i \leq r$, we either have $X_i = \{x_i\}$ for some $x_i \in \mathbb{Z}$ or $X_i = \mathbb{Z}$. Moreover, if $\pi_{i,r}(V)$ is not a singleton for some $1 \leq i \leq r$, then we call $\pi_{i,r}$ to be a \emph{coordinate map} on $V$. Finally, given a finite subset $X$ of some axis aligned subspace $V$, we will now define its \emph{skew-dimension} $d_*(X)$. If $|X| = 1$, then we denote $\dim_*(X) = 0$. Otherwise, let $1 \leq i \leq r$ be the largest number such that $\pi_{i,r}$ is a coordinate map on $V$ and $|\pi_{i,r}(A)| >1$. In this case, we define
\[ d_*(A) = 1 + \max_{x \in \mathbb{Z}} \dim_*( \pi_{i,r}^{-1}(x) \cap A). \]
\par

With this notation in hand, we are now ready to prove Lemma $\ref{cmb}$.

\begin{proof}[Proof of Lemma $\ref{cmb}$]
Since $A, X \subseteq \mathbb{N}$ are finite sets, we note that the set 
\[ \mathcal{P}  = \{ p \ \text{prime} \ | \ p \ \text{divides} \ y \ \ \text{for some} \ y \in A \cup X \}, \]
is finite, whence, writing $\mathcal{P} = \{ p_1, \dots, p_r\}$ for some $r \in \mathbb{N}$, we define the map $\psi : \mathbb{N} \to \mathbb{Z}^d$ such that $\psi(y) = (\nu_{p_1}(y), \dots, \nu_{p_r}(y))$. For the purposes of this proof, let $\pi_i : \mathbb{Z}^r \to \mathbb{Z}$ be the projection map defined as $\pi_i(x_1, \dots, x_r) = x_i$ for every $1 \leq i \leq r$. Our hypothesis now implies that $|\psi(A) + \psi(X) + \psi(X)| \leq K |\psi(X)|$, and so, writing $V$ to be the largest binary set contained in $\psi(A)$, we use \cite[Theorem 2.7]{MRSZ2020} to deduce that
\[ |V| \leq |V + \psi(X) + \psi(X)| |\psi(X)|^{-1} \leq |\psi(A) +  \psi(X) + \psi(X)| |\psi(X)|^{-1} \leq K. \]
But now, we may employ \cite[Proposition 8.4.2]{Gr2022} to infer that there exists $A' \subseteq A$ such that
\[ |A'| \geq |A|/|V| \geq |A|/K \ \ \text{and} \ \ d_*(\psi(A')) \leq \log |V| \leq \log K. \]
Noting the definition of skew-dimension and query complexity,  it is relatively straightforward to show that $q(A') \leq d_{*}(\psi(A')) +1$.  Combining this with the preceding discussion dispenses the desired conclusion.
\end{proof}

A key ingredient in our proof of Theorem $\ref{th3}$ will be iterative applications of Theorem $\ref{th46}$. In particular, assuming our set $A$ to have a large multiplicative energy, we will apply Theorem $\ref{th46}$ to extract a large subset of $A$ that exhibits various other types of multiplicative structure. Moreover, we keep repeating this argument until the remaining set has a small multiplicative energy. Here, it will be important to be able to control the number of steps that such an iterative process takes, a task for which we will employ \cite[Lemma 4.2]{Mu2021c}. We present this below.

\begin{lemma} \label{tym}
Let $0 < c < 1$ and $C>0$ be constants. Let $A_0 = A$, and for each $i \geq 1$, define $A_i = A_{i-1} \setminus U_i$ where $U_i$ is some set satisfying $|U_i| \geq C|A_{i-1}|^{1-c}$. Then, for some $r \leq 2 (\log |A| + 2) + C^{-1} \frac{|A|^{c}}{2^c - 1}$, we must have $|A_r| \leq 1$. 
\end{lemma}


\section{Proof of Theorem $\ref{th3}$}

We dedicate this section to proving Theorem $\ref{th3}$.  Let $\mathcal{D}$ be some large constant that we will fix later, and let $s$ be some natural number sufficiently large in terms of $\mathcal{D}$. We define 
\begin{equation} \label{tktk}
k = \ceil{ \log s / (\mathcal{D} \log \log s) } 
\end{equation}
We first begin with assuming that our set $A$ is a subset of $\mathbb{N} \setminus \mathcal{Z}_{\vec{\varphi}}$. Furthermore, we may assume that $M_{s}(A) > |A|^{2s - k}$, since otherwise, we are done.
\par

As previously mentioned, we now perform an iteration, where at each step, we extract a large subset of our set $A$ which satisfies suitable arithmetic properties. We first set $A_0 = A$. Next, for every $i \in \mathbb{N}$, we will begin the $i^{th}$ step of our iteration with some subset $A_{i-1} \subseteq A_0$. If $M_{s}(A_{i-1}) \leq |A_{i-1}|^{2s - k}$, we stop our iteration, else, we apply a multiplicative version of Theorem $\ref{th46}$ to obtain a finite, non-empty subset $U_i \subseteq \mathbb{N}$ satisfying
\[ |U_i| \ll |A_{i-1} |^k \ \ \text{and} \ \ \max_{x \in U_i \cdot A_{i-1}^{-1} } |A _{i-1} \cap x^{-1} \cdot U_i| \gg |A_{i-1}|^{ 1- k \mathcal{C}/ \log s} , \]
such that
\[ |U_i^{(m)} / U_i^{(n)} | \ll_{m,n} |A_{i-1}|^{k \mathcal{C}(m+n) / \log s} |U_i| \]
holds true for every $m,n \in \mathbb{N} \cup \{0\}$. In particular, writing $A_{i-1}' =  A_{i-1} \cap x^{-1} \cdot U_i$, for some $x \in U_i \cdot A_{i-1}^{-1}$ which maximises $|A_{i-1} \cap x^{-1} \cdot U_i|$, we get that $|A_{i-1}'| \gg |A_{i-1}|^{ 1- k \mathcal{C}/ \log s}$. Thus we see that
 \[ |A_{i-1}' \cdot U_i \cdot U_i| \leq |U_i^{(3)}| \ll |A_{i-1}|^{k \mathcal{C} / \log s} |U_i| ., \]
which then combines with Lemma $\ref{cmb}$ to give us a subset $B_i \subseteq A_{i-1}'$ such that
  \begin{equation} \label{apr1}
 |B_i| \gg |A_{i-1}'|  |A_{i-1}|^{-k \mathcal{C} / \log s} \gg |A_{i-1}|^{1 - k \mathcal{C}/ \log s} \ \ \text{and} \ \ q(B_i) \leq \log (|A_{i-1}|^{k \mathcal{C} / \log s}) + O(1). 
 \end{equation}
 Moreover, since $B_i \subseteq x^{-1} \cdot U_i$, we may also deduce that
 \begin{equation} \label{apr2}
 |B_i^{(m)}/B_i^{(n)}| \ll_{m,n} |A_{i-1}|^{k \mathcal{C}(m+n) / \log s} |U_i|  \ll |A_{i-1}|^{k \mathcal{C}(m+n) / \log s} |A|^k, 
 \end{equation}
 for every $m, n \in \mathbb{N} \cup \{0\}$. We now set $A_i = A_{i-1} \setminus B_i$ and proceed with the $(i+1)^{th}$ step of our algorithm. By way of Lemma $\ref{tym}$, we must have $|A_r| \leq 1$ for some 
 \[ r \ll_{k, s} |A|^{k \mathcal{C}/ \log s} ,\]
 in which case, we would trivially have $M_{s}(A_r) \leq |A_r|^{2s - k}$, thus terminating the algorithm. This gives us a partition of the set $A$ as $A = B_1 \cup \dots \cup B_{r} \cup A_r$, where the sets $A_r, B_1, \dots, B_r$ are pairwise disjoint. We set $C = A_r$ and $B = B_1 \cup \dots \cup B_r$.
 \par
 
 We first proceed to prove the bound on the mixed multiplicative energy $M_{s, \vec{\varphi}}(B)$, and so, just for this part of the proof, we assume that $\varphi_j(0) \neq 0$ for $1 \leq j \leq 2s$. Writing $q = 10k$, this allows us to employ Lemma $\ref{tst}$ to see that
 \[ M_{q, \varphi_j}(B_i) \ll_{q,d} |B_i^{(q)}/B_i^{(q)}|  (d+3)^{16q(B_i) q}  (2q)^{2q(B_i) q} |B_i|^q,\]
 for each $1 \leq i \leq r$ and $1 \leq j \leq 2s$. We may combine this with $\eqref{apr1}$ and $\eqref{apr2}$ so as to obtain the bound
 \[ M_{q, \varphi_j}(B_i) \ll_{q,d} |A|^{k (1 + q \mathcal{C} / \log s) }  |A|^{q k \mathcal{C}_d  \log q / \log s} |B_i|^{q} . \]
Next, we use the remark following Lemma $\ref{wm}$ along with the preceding inequalities to deduce that
 \[ M_{q, \varphi_j}(B)  \ll_{d,q} r^{2q} \sup_{1 \leq i \leq r} M_{q, \varphi_j}(B_i)  \ll_{d,q,s} |A|^{k q \mathcal{C}/ \log s}   |A|^{k (1 + q\mathcal{C}_d  \log q / \log s) }  |B|^{q}    .\]
 It is worth noting that 
\[ |B| \geq |B_1| \gg |A_0|^{1 - k \mathcal{C}/ \log s} = |A|^{1 - k \mathcal{C}/ \log s} \geq |A|^{1/2}, \]
 whenever $\mathcal{D}$ is sufficiently large in terms of $\mathcal{C}$. Thus, the estimate
 \[ M_{q, \varphi_j}(B) \ll_{d,q} |B|^{2k} |B|^{k q \mathcal{C}_d \log q / \log s } |B|^q \ll_{d,q} |B|^{q + 2q/5} \leq |B|^{2q-k}, \]
holds true whenever $\mathcal{D}$ is sufficiently large in terms of $\mathcal{C}_d$.  Noting the fact that for any $x \in \mathbb{R}$ and any $\varphi \in \mathbb{Z}[x]$ satisfying $1 \leq \deg \varphi \leq d$,  there are at most $O_d(1)$ solutions to $x = \varphi(b)$ with $b \in B$,  we deduce that
\[ M_{s, \vec{\varphi}}(B)  \ll_{s,d} M_{s}(\varphi_1(B), \dots, \varphi_{2s}(B)) . \]
This, combines with the preceding discussion,  to give us
 \begin{align*}
  M_{s, \vec{\varphi}}(B) 
  &  \ll_{s,d} M_{s}(\varphi_1(B), \dots, \varphi_{2s}(B)) \ll_{s,d}  \prod_{j=1}^{2s} M_{s}(\varphi_j(B))^{1/2s} \\
  & \leq \prod_{j=1}^{2s} ( |B|^{2s - 2q} M_{q, \varphi_j}(B) )^{1/2s} \ll_{d,s} |B|^{2s - k},
  \end{align*}
where the second and third inequalities follow from the analogues of Proposition \ref{gvup} and Lemma \ref{wm} mentioned in the remark at the end of \S3 respectively.
 \par

We now turn to the case of dealing with the mixed additive energy $E_{s, \vec{\varphi}}(B)$. This, in fact, proceeds in a very similar fashion to the proof for the upper bound on $M_{s, \vec{\varphi}}(B)$, just with the applications of Lemmata $\ref{tst}$ and $\ref{wm}$ replaced by applications of Lemmata $\ref{ayay}$ and $\ref{reprove}$ respectively. In fact, the case of $E_{s, \vec{\varphi}}(B)$ is slightly simpler than the situation where we provide estimates for $M_{s, \vec{\varphi}}(B)$, since the former does not require utilising the upper bound $\eqref{apr2}$ on the many-fold product sets of $B_1, \dots, B_r$.
\par

 Thus, we have proven Theorem $\ref{th3}$ when $A$ is a finite subset of $\mathbb{N}\setminus \mathcal{Z}_{\vec{\varphi}}$ and $\vec{\varphi} \in (\mathbb{Z}[x])^{2s}$. We now reduce the more general case when $A$ is a finite subset of $\mathbb{Q}$ and $\vec{\varphi} \in (\mathbb{Q}[x])^{2s}$ to the aforementioned setting.  As in \S3, we see that upon dilating the set $A$ appropriately, we may reduce the more general case to the setting when $A \subseteq \mathbb{Z}$ and $\vec{\varphi} \in (\mathbb{Z}[x])^{2s}$.  We now focus on the multiplicative setting of Theorem $\ref{th3}$ for the latter case,  and so,  given some finite $A \subseteq \mathbb{Z}$ and some $\vec{\varphi} \in (\mathbb{Z}[x])^{2s}$ such that $\varphi_{j}(0) \neq 0$ for every $1 \leq j \leq 2s$, we define $\vec{\varphi}'$ to satisfy $\varphi_j'(x)= \varphi_j(-x)$ for every $x \in \mathbb{Z}$ and $1 \leq j \leq 2s$. Next, we write $\mathcal{Z} = (\mathcal{Z}_{\vec{\varphi}} \cup \mathcal{Z}_{\vec{\varphi}'})\setminus \{0\}$ and we denote
 \[ A_1 = (A \cap (0, \infty)) \setminus \mathcal{Z} \ \ \text{and} \ \ A_2 = (A \cap (-\infty,0)) \setminus \mathcal{Z} \ \ \text{and} \ \ A_3 = A \cap \mathcal{Z}\ \ \text{and} \ \ A_4 = A \cap \{0\}. \]
 By Theorem $\ref{th3}$, we see that $A_1 = B_1 \cup C_1$ and $-A_2 = (-B_2) \cup (-C_2)$ such that for every $i \in \{1,2\}$, we have $B_i \cap C_i = \emptyset$ as well as
 \[ M_{s}(C_1) \ll_{s,d} |C_1|^{2s - \eta_s}  \ \ \text{and} \ \   M_{s, {\varphi}_j}(B_1) \ll_{s,d} |B_1|^{2s - \eta_s} \ \ (1 \leq j \leq 2s)  \]
 and
 \[ M_{s}(-C_2) \ll_{s,d} |C_2|^{2s- \eta_s}   \ \ \text{and} \ \  M_{s, {\varphi}_j'}(-B_2) \ll_{s,d} |B_2|^{2s - \eta_s}  \ \ (1 \leq j \leq 2s) ,\]
 for some $\eta_s \gg_{d} \log s/ \log \log s$. Noting that $M_{s, {\varphi}_j'}(-B_2) = M_{s, {\varphi}_j}(B_2)$ for every $1\leq j\leq 2s$ and $|A_3|,|A_4| \ll_{d} 1$, we may now write $B = B_1 \cup B_2 \cup A_4$ and $C = C_1 \cup C_2 \cup A_3$, whereupon, the remark following Lemma $\ref{wm}$ implies that
 \[ M_{s, \vec{\varphi}}(B) \ll_{s,d}   \max_{1\leq j \leq 2s} \max\{ M_{s,\varphi_j}(B_1) , M_{s, \varphi_j}(B_2)   \} \ll_{s,d}   |B|^{2s - \eta_s} \ \ \text{and} \ \ M_{s}(C) \ll_{s,d} |C|^{2s - \eta_s}. \]
 A similar strategy maybe followed to resolve the case when $M_{s, \vec{\varphi}}(B)$ is replaced by $E_{s, \vec{\varphi}}(B)$. This concludes our proof of Theorem $\ref{th3}$.

%

\section{Proof of Theorems $\ref{mve}$ and $\ref{fin6}$}

Our aim in this section is to prove Theorem $\ref{mve}$ and its multiplicative variants. We begin by recording the following greedy covering lemma from \cite[Lemma $5$]{HP2021}.
%
%

\begin{lemma} \label{hph}
Let $A, B$ be subsets of $\mathbb{R} \setminus \{0\}$ such that $|A| \geq 2$ and $|A \cdot B| \leq C|B|$, for some $C \geq 1$. Then there exists a set $S \subseteq A \cdot B^{-1}$ with $|S| \ll C \log |A|$ such that $A \subseteq S \cdot B$. 
\end{lemma}


We are now ready to prove Theorem $\ref{mve}$.


\begin{proof}[Proof of Theorem $\ref{mve}$]
We begin noting that Theorem $\ref{mve}$ holds trivially when $|A| = 1$,  and so,  we may assume that $|A| \geq 2$.  We now apply Lemma $\ref{pr21}$ to infer that $|A^{(3)}| \leq K^3 |A|$, whence, we may apply Lemma $\ref{cmb}$ to deduce the existence of a set $A' \subseteq A$ with $|A'| \geq  |A|/K^3$ and $q(A') \leq 3 \log K + O(1)$. This implies that $|A \cdot A'| \leq K|A| \leq K^{4} |A'|$, whenceforth, we may apply Lemma $\ref{hph}$ to obtain a set $S \subseteq \mathbb{Q} \setminus \{0\}$ such that $|S| \ll K^{4} \log |A|$ and $A \subseteq S \cdot A'$. In particular, let $S = \{s_1, \dots, s_r\}$, where $r = |S|$, and let $A_i = A \cap s_i \cdot A'$.  Using orthogonality and applying H\"{o}lder's inequality as in \eqref{hl3},  we may deduce that 
\begin{align} \label{list5}
E_{s, \frak{a}, \varphi}(A)   
 & =     \int_{[0,1)} | \sum_{a \in A} \frak{a}(a) e( \alpha \varphi(a) ) |^{2s} d \alpha \nonumber    \\
&   \leq r^{2s} \sup_{1 \leq i \leq r}   \int_{[0,1)} | \sum_{a \in A_i} \frak{a}(a) e( \alpha \varphi(a) ) |^{2s} d \alpha  \nonumber   \\
& =   r^{2s} \sup_{1 \leq i \leq r}       E_{s, \frak{a}, \varphi}(A_i). 
\end{align}
Finally, since for each $1 \leq i \leq r$, the set $A_i$ lies in a dilate of the set $A'$, we must have $q(A_i) \leq q(A') \leq 3 \log K + O(1)$. This allows us to apply Lemma $\ref{ayay}$ to deduce that
\begin{align*}
 E_{s, \frak{a}, \varphi}(A)^{1/s} 
 & \ll_{s,d} (K^{4} \log |A|)^{2} (d^2+2)^{12 \log K} (2s)^{6 \log K} \sum_{a \in A} \frak{a}(a)^2 \\
 & \ll_{s,d} K^{C}(\log |A|)^2 \sum_{a \in A} \frak{a}(a)^2,
 \end{align*}
where $C= 8 + 12 \log (d^2 + 2) + 6 \log (2s)$. 
\end{proof}

The proof of Theorem $\ref{fin6}$ follows mutatis mutandis, and we make some brief remarks concerning this below. Adapting the methods of \S5 along with the results of this section, that is, following ideas from \S5 while replacing any usage of Lemma $\ref{chang}$ with Lemma $\ref{ob2}$ as well as substituting the application of H\"{o}lder's inequality in $\eqref{list5}$ with an application of Lemma $\ref{wm}$, we may obtain a similar result for $J_{s, \frak{a}, {\varphi}}(A)$, where $\varphi \in \mathbb{Q}[x]$ with $\varphi(0) \neq 0$. We may further use a combination of Lemmata $\ref{tst}$ and $\ref{pr21}$ in place of Lemma $\ref{ob2}$ to obtain an estimate for $M_{s, \frak{a}, {\varphi}}(A)$, when $|A\cdot A| \leq K |A|$. This finishes the proof of Theorem $\ref{fin6}$ when $\vec{\varphi} = (\varphi, \dots, \varphi)$ for some $\varphi \in \mathbb{Q}[x]$ with $\varphi(0) \neq 0$.  The more general case may then be deduced from this special case by employing Proposition $\ref{gvup}$ and the ideas involved therein.
\par

\section{Additive and multiplicative Sidon sets}

We dedicate this section to proving Theorem $\ref{sid1}$ and Proposition $\ref{bwex}$. In our proof of Theorem $\ref{sid1}$, we will use the following amalgamation of \cite[Lemmata 5.1 and 5.2]{JM2022}.

\begin{lemma} \label{brns} 
Given a finite set $A \subseteq \mathbb{N}$ and a natural number $s \geq 2$ and some $c>0$, if $E_{s}(A) \ll_{s} |A|^{2s - 2 + 1/s - c}$, then there exists a $B_{s}^{+}[1]$ set $X \subseteq A$ with $|X| \gg_{s} |A|^{1/s + c/2s}$. Similarly, if we have $M_{s}(A) \ll_{s} |A|^{2s - 2 + 1/s - c}$, then there exists a $B_{s}^{\times}[1]$ set $Y \subseteq A$ satisfying $|Y| \gg_{s} |A|^{1/s + c/2s}$
\end{lemma} 

With this result in hand, we now present the proof of Theorem $\ref{sid1}$. 

\begin{proof}[Proof of Theorem $\ref{sid1}$]
We begin by applying Theorem $\ref{th3}$ to obtain a decomposition $A = B \cup C$, with $B, C$ being disjoint and satisfying the conclusion of Theorem $\ref{th3}$. We divide our proof into two cases, the first of these being if $|C| \geq |A|/2$. Writing $C_1 = C \cup (0, \infty)$ and $C_2 = C \cup (-\infty, 0)$, we see that either $|C_1| \geq |C|/3$ or $|C_2| \geq |C|/3$. If the first inequality holds, then we see that
\[ M_{s}(C_1) \leq M_{s}(C) \ll_{s} |C|^{2s - \eta_s} \ll_{s} |C_1|^{2s - \eta_s},\]
whence, applying Lemma $\ref{brns}$ delivers a $B_{s}^{\times}[1]$ subset $Y \subseteq C$ such that
\[ |Y| \gg_{s} |C_1|^{\eta_s/4s} \gg_{s} |A|^{\eta_s/4s} .  \]
On the other hand, if $|C_2| \geq |C|/3$, then we may again employ Lemma $\ref{brns}$ to obtain a large $B_{s}^{\times}[1]$ subset $Y'$ of $-C_2$, which, in turn, gives us a large $B_{s}^{\times}[1]$ subset $-Y'$ of $C_2$.
\par

On the other hand, if $|B| \geq |A|/2$, then we may define $B_1 = \{ b \in B \ | \ \varphi(b) \in (0, \infty)\}$ and $B_2 = \{ b \in B \  \ \varphi(b) \in (- \infty,0)\}$. As before, either $|B_1| \geq |B|/3$ or $|B_2| \geq |B|/3$. If the first inequality holds, then we have that
\[  E_{s}(\varphi(B_1))    \leq  E_{s, \varphi}(B_1)  \ll_{s,d} |B|^{2s - \eta_s} \ll_{s,d} |B_1|^{2s - \eta_s} \ll_{s,d} |\varphi(B_1)|^{2s - \eta_s} , \]
whence Lemma $\ref{brns}$ yields a $B_{s}^{+}[1]$ set $X' \subseteq \varphi(B_1)$ such that $|X'| \gg_{s} |\varphi(B_1)|^{\eta_s/4s}$. For every $x \in X'$, fix some $b_x \in B$ such that $\varphi(b_x) = x$ and let $X = \{ b_x \ | \ x \in X' \}$. Then we have that 
\[ |X| = |X'| \gg_{s} |\varphi(B_1)|^{\eta_s/4s} \gg_{s,d} |B_1|^{\eta_s/4s} \gg_{s,d} |A|^{\eta_s/4s}, \]
and so, we finish the subcase when $|B_1| \geq |B|/3$. We may proceed similarly in the case when $|B_2| \geq |B_1|/3$ to finish the proof of additive part of Theorem $\ref{sid1}$. As for the multiplicative case, since $\varphi(0) \neq 0$, we see that $M_{s, \varphi}(B) \ll_{s,d} |B|^{2s - \eta_s}$. We may now continue as in the additive case to finish our proof of Theorem $\ref{sid1}$. 
\end{proof}

Thus, we have shown, in the form of Theorem $\ref{sid1}$, that strong low energy decompositions deliver large additive and multiplicative Sidon sets. We will now show that, roughly speaking, such an implication may be reversed as well. 

\begin{Proposition} \label{rev4}
Let $s \geq 2$, let $\vec{\varphi} \in (\mathbb{Z}[x])^{2s}$ satisfy $\varphi_1 = \dots = \varphi_{2s} = \varphi$, for some $\varphi \in \mathbb{Z}[x]$ with $\deg \varphi = d \geq 1$ and $\varphi(0) \neq 0$, and let $\delta_s$ be defined as in Theorem $\ref{sid1}$. Then, given any finite set $A \subseteq \mathbb{Z}$, we can find disjoint sets $B, C \subseteq A$ such that $A = B \cup C$ and
\[ \max\{ E_{s, \vec{\varphi}}(B), M_{s, \vec{\varphi}}(B), M_{s}(C)\} \ll_{s} |A|^{2s - \delta_s} . \]
\end{Proposition}

\begin{proof}
Let $A$ be a finite set of integers. We may iteratively apply Theorem $\ref{sid1}$ to the set $A$ to obtain an absolute constant $D = D_{d,s}>0$ and sets $A =A_0 \supseteq A_1 \supseteq  \dots \supseteq A_r$, with $|A_{i-1} \setminus A_i| = D |A_{i-1}|^{\delta_s/s}$ for every $1 \leq i \leq r$, where the set $A_{i-1} \setminus A_i$ is either a $B_{s}^{+}[1]$ set or a $B_{s}^{\times}[1]$ set. A straightforward application of Lemma $\ref{tym}$ implies that $|A_r| \leq 1$ for some $r \ll_s |A|^{1 - \delta_s/s}$. Thus, we may partition $A$ as 
\[ A = (\cup_{1 \leq i \leq r_1} X_i )\cup (\cup_{1 \leq j \leq r_2} Y_j),\]
 where $X_i$ is a $B_{s, \varphi}^{+}[1]$ set for every $1 \leq i \leq r_1$, the set $Y_j$ is a $B_{s}^{\times}[1]$ set for every $1 \leq j \leq r_2$, the sets $X_1, \dots, X_{r_1}, Y_1, \dots, Y_{r_2}$ are pairwise disjoint and $r_1 + r_2 \ll_s |A|^{1 - \delta_s/s}$. Writing $B = \cup_{1 \leq i \leq r_1} X_i$, we may apply Lemma $\ref{reprove}$ to deduce that
 \[ E_{s, \vec{\varphi}}(B) \ll r_1^{2s} \sup_{1 \leq i \leq r_1} E_{s, \vec{\varphi}}(X_i) \ll_{s,d} |A|^{2s - 2\delta_s} \sup_{1 \leq i \leq r_1} |X_i|^{s}  \ll_{s,d} |A|^{2s - \delta_s} .   \]
One may similarly write $C = \cup_{1 \leq j \leq r_2} Y_j$ and apply Lemma $\ref{wm}$ suitably to obtain the bound $M_{s}(C) \ll_{s,d} |A|^{2s - \delta_s}$. Finally, applying this circle of ideas mutatis mutandis, with $M_{s, \vec{\varphi}}(B)$ replacing $E_{s, \vec{\varphi}}(B)$,  allows us to conclude the proof of Proposition $\ref{rev4}$.
\end{proof}


We end this section by recording the proof of Proposition $\ref{bwex}$.

%

\begin{proof}[Proof of Proposition $\ref{bwex}$]
Let $d,s \in \mathbb{N}$ satisfy $s \geq 10d(d+1)$ and write 
\[ A:=A_{m, n} = \{ (2i+1) \cdot 2^j \ | \ 1 \leq  i \leq m \  \text{and} \ 1 \leq j \leq n \}, \]
for some $n,m \in \mathbb{N}$ to be fixed later. We will first show that for appropriate choices of $m,n$, every $B \subseteq A$ with $|B| \geq |A|/2$ and every $\varphi \in \mathbb{Z}[x]$ with $\deg \varphi = d$ satisfy
\[ E_{s, \vec{\varphi}}(B) , M_{s}(B) \gg_{s, \vec{\varphi}} |A|^{s + s/3 - (d^2 + d + 2)/6} , \]
where $\vec{\varphi} = (\varphi, \dots, \varphi)$. We begin by observing that $|B^{(s)}| \leq |A^{(s)}| \ll_{s} m^{s} n$, which, together with $\eqref{csap}$, gives us 
\begin{equation} \label{onlyu}
 M_{s}(B) \geq |B|^{2s} |B^{(s)}|^{-1} \gg_{s} m^{s} n^{2s-1} . 
 \end{equation}
 \par
 
We now turn to analysing $E_{s, \vec{\varphi}}(B)$, and so, we write ${S}_j = \{ 2^j \cdot (2i+1) \ | \ 1 \leq i \leq m \}$ and note that $A = S_1 \cup \dots \cup S_n$ as well as that $|A| = mn$. Next, setting
\[ \mathcal{M} = \{ j \in \mathbb{N} \ | \ |B \cap S_j| \geq |B|/4n \}, \]
we deduce that
\begin{equation} \label{ny1}
 \sum_{j \in \mathcal{M}} |B \cap S_j| = |B| - \sum_{j \notin \mathcal{M}} |B \cap S_j| > |B| - |B|/4 \geq 3|B|/4. 
 \end{equation}
Furthermore, writing $B_j = B \cap S_j$ for every $j \in \mathcal{M}$ and $\varphi(x) = a_0 + a_1 x + \dots + a_d x^d$ for some $a_0, \dots, a_d \in \mathbb{Z}$ with $a_d \neq 0$, we discern that the set $s \varphi (B_j)$ is contained in $s \varphi (S_j)$, which itself is a subset of the set
\[ T =  \{ sa_0 + a_1 2^j (x_1 + \dots + x_s) + \dots + a_d 2^{jd}(x_1^d + \dots + x_s^d) \ | \ x_1, \dots, x_s \in \{3,5, \dots, 2m+1\}  \}.  \]
In particular, we have that
\[ |T| \leq \prod_{i=1}^{d}  (s |a_i| (2m+1)^i) \ll_{s, \varphi} m^{d(d+1)/2} ,  \]
whence, applying $\eqref{csap}$ again gives us
\[ E_{s, \varphi}(B_j) \geq E_{s}(\varphi(B_j)) \gg_{d,s}  |B_j|^{2s} |s \varphi (B_j)|^{-1} \gg_{d,s} |B_j| (|B|/4n)^{2s-1}  m^{-d(d+1)/2} .    \]
Combining this with $\eqref{ny1}$ then allows us to infer that
\[ E_{s}(B) \geq \sum_{j \in \mathcal{M}} E_{s}(B_j) \gg_{d,s} (|B|/4n)^{2s-1} m^{-d(d+1)/2} \sum_{j \in \mathcal{M}} |B_j| \gg_{d,s} m^{2s  - d(d+1)/2} n .    \]
We may now optimise this with the inequality recorded in $\eqref{onlyu}$, which gives us that $m^{s - d(d+1)/2} =  n^{2s - 2}$, that is, $n^{3s - d(d+1)/2 - 2} = |A|^{s - d(d+1)/2}$. Here, choosing $n$ to be appropriately large in terms of $N$ already gives us $|A| \gg N$. Moreover, substituting the value of $n$ into the preceding set of inequalities yields the bound
\[ E_{s, \vec{\varphi}}(B), M_{s}(B) \gg_{s} |B|^{s} |B|^{ \frac{(s - d(d+1)/2)(s-1)}{3s - d(d+1)/2 - 2}}. \]
An elementary computation now shows that
\[ \frac{(s - d(d+1)/2)(s-1)}{3s - d(d+1)/2 - 2} \geq \frac{s-1}{3} \bigg( 1  - \frac{d (d+1)}{2s} \bigg) \geq \frac{s}{3}  - \frac{d(d+1) + 2}{6}  \]
whenever $s \geq 10d(d+1)$, consequently proving the first part of Proposition $\ref{bwex}$.
\par
 
For our second part, let $s \geq 2$ be an even integer and let $P$ be the set consisting of the first $\ceil{N^{\frac{s}{2s+2}}}$ prime numbers and let $Q$ be the set consisting of the next $\ceil{N^{\frac{s+2}{2s+2}}}$ prime numbers. We further define $A' := A_{N}' = P \cdot Q,$ whence, we have the estimate $|A'| \gg N$. In \cite[Section 2]{JM2022}, applying graph theoretic results from \cite{NV05}, it was shown that the largest $B_{s}^{\times}[1]$ subset $Y$ of $A'$ satisfies $|Y| \ll_{s} |A'|^{\frac{1}{2} + \frac{1}{s+2}}$, thus, we only focus on the additive case here. By way of the prime number theorem, we may deduce that $A' \subseteq \{1, \dots, M\}$ for some $M \ll N (\log N)^2$, whence, writing $X$ to be the  largest $B_{s, \varphi}^{+}[1]$ subset of $A'$, we get that
 \[|X|^{s} \ll_{s,d} |s \varphi(X)| \leq |s \varphi(A')| \leq |s \varphi(\{1, \dots, M\})| \ll_{s, \varphi} M^{d} \ll_{s,d} N^{d} (\log N)^{2d} . \]
This dispenses the desired bound, and so, we conclude our proof of Proposition $\ref{bwex}$.
\end{proof}


\bibliographystyle{amsbracket}
\providecommand{\bysame}{\leavevmode\hbox to3em{\hrulefill}\thinspace}

\end{document}